\documentclass[pdf, 10pt]{amsart}

\usepackage{latexsym, amsthm, amsmath, amssymb}
\usepackage[OT1]{fontenc}
\usepackage[colorlinks,citecolor=blue,urlcolor=blue]{hyperref}
\usepackage{color}

\newtheorem{thm}[equation]{Theorem}

\newtheorem{prop}[equation]{Proposition}

\newtheorem{lem}[equation]{Lemma}

\newtheorem{rem}[equation]{Remark}

\def\R{{\mathbb{R}}}

\def\N{{\mathbb{N}}}
\def\Z{{\mathbb{Z}}}
\def\E{{\mathbb{E}}}
\def\P{{\mathbb{P}}}

\def\B{{\mathcal{B}}}
\def\Cc{{\mathcal{C}}}

\def\H{{\mathcal{H}}}

\def\bl1{\mathcal{B}_{\rho,\epsilon,\lambda}}

\newcommand{\Id}{\mbox{Id}}

\newcommand{\supp}{\mbox{supp}}

\newcommand{\Gl}{\mbox{Gl}}
\newcommand{\al}{\alpha}
\newcommand{\be}{\beta}
\newcommand{\ga}{\gamma}
\newcommand{\de}{\delta}
\newcommand{\ka}{\kappa}
\newcommand{\la}{\lambda}
\newcommand{\La}{\Lambda}

\newcommand{\eps}{\varepsilon}

\renewcommand{\c}{\chi}
\newcommand{\si}{\sigma}

\newcommand{\te}{\theta}
\newcommand{\Te}{\Theta}
\newcommand{\8}{\infty}

\newcommand{\Rd}{{\mathbb{R}^d}}

\newcommand{\Sd}{{\mathbb{S}^{d-1}}}

\renewcommand{\l}{\left}
\renewcommand{\r}{\right}
\newcommand{\iss}[2]{\left\langle #1,#2\right\rangle}
\newcommand{\limx}[2]{\lim_{#1\rightarrow#2}}
\newcommand{\limn}{\lim_{n\rightarrow\infty}}

\newcommand{\sumk}[2]{\sum_{k=#1}^{#2}}
\newcommand{\sumi}[2]{\sum_{i=#1}^{#2}}
\newcommand{\sumij}[2]{\sum_{i, j=#1}^{#2}}

\newcommand{\Q}{\mathbb{Q}}
\renewcommand{\O}{\Omega}
\renewcommand{\o}{\omega}

\newcommand{\gq}{ ^a \!Q^s}

\newcommand{\gt}{ ^a \!\theta}
\newcommand{\go}{ ^a \! \Omega}
\newcommand{\z}{\zeta}
\newcommand{\Sp}{{\mathbb{S}^{+}}}
\newcommand{\ra}{\cdot}
\newcommand{\raa}{\cdot\ldots\cdot}
\newcommand{\I}[1]{\mathbf{1}_{#1}}

\newcommand{\krr}{,\ldots,}
\newcommand{\ov}[1]{\overline{#1}}

\numberwithin{equation}{section}

\begin{document}

\title{On fixed points of a generalized multidimensional affine recursion}
\author[M. Mirek]
{Mariusz Mirek}
\address{M.Mirek\\
University of Wroclaw\\ Institute of Mathematics\\
 50-384 Wroclaw\\
 Poland}
\email{mirek@math.uni.wroc.pl}

\thanks{
This research project was partially supported by MNiSW grant N N201 392337.}

\begin{abstract}
Let $G$ be a multiplicative subsemigroup of the general linear
group $\Gl(\mathbb{R}^d)$ which consists of matrices with positive entries such that every column and
every row contains a strictly positive element. Given a
$G$--valued random matrix $A$, we consider the following
generalized multidimensional affine equation
\begin{align*}
    R\stackrel{\mathcal{D}}{=}\sum_{i=1}^N A_iR_i+B,
\end{align*}
where $N\ge2$ is a fixed natural number, $A_1,\ldots,A_N$ are independent copies of $A$, $B\in\mathbb{R}^d$ is
a random vector with positive entries, and $R_1,\ldots,R_N$ are independent copies of $R\in\mathbb{R}^d$,
which have also positive entries. Moreover, all of them are mutually independent and $\stackrel{\mathcal{D}}{=}$ stands for the equality in distribution. We will show with the aid of spectral theory developed by Guivarc'h and Le Page \cite{gaGL}, \cite{gaGL1} and Kesten's renewal theorem \cite{K1}, that under
appropriate conditions, there exists $\chi>0$ such that $\P(\{\langle R, u\rangle>t\})\asymp t^{-\chi},$ as
$t\to\8$, for every unit vector $u\in\mathbb{S}^{d-1}$ with positive entries.
\end{abstract}
\maketitle

\section{Introduction and statement of the results}\label{INT}
We consider the Euclidean space $\Rd$ endowed with the scalar product $\iss x
y=\sum_{i=1}^{d}x_{i}y_{i}$, the norm $|x|=\sqrt{\langle x, x\rangle},$ and its Borel $\si$--field
$\B or(\R^d)$. We say that $\Rd\ni x=(x_1,\ldots,x_d)\ge0$ is positive (resp. $\Rd\ni
x=(x_1,\ldots,x_d)>0$ is strictly positive), when $x_n\ge 0$, (resp. $x_n> 0$) for every $1\le n\le
d$. By $\R^d_+$ we denote the set of all positive vectors, and we define the set
$\Sp=\R^d_+\cap\Sd$ of all positive vectors on the unit sphere $\Sd=\{x\in\Rd: |x|=1\}$ with the
distance being the restriction of the Euclidean norm to $\Sp$. Given $x\in\Rd$ we denote its
projection on $\Sd$ by $\ov x=\frac{x}{|x|}$.

 Let $\Gl(\Rd)$ be the group of $d\times d$ invertible
matrices on $\Rd$, with the operator norm $\|\cdot\|$ associated with the Euclidean norm $|\cdot|$
on $\Rd$, i.e. $\|a\|=\sup_{x\in\Sd}|ax|$ for every $a\in\Gl(\Rd)$.

Suppose that $G$ is a multiplicative subsemigroup of $\Gl(\Rd)$ which
consists of matrices
with positive entries such that every column and every row
contains a strictly positive element. By $G^{\circ}$ we denote the
multiplicative subsemigroup of $G$ composed of matrices with
strictly positive entries. It is easy to see that $G$ provides a
projective action on $\Sp$ which is given by
$$G\times\Sp\ni(a, x)\mapsto a\ra x=\frac{ax}{|ax|}\in\Sp.$$


Let $A$ be a $G$--valued random matrix distributed according to a
probability measure $\mu $ on $G$, and $B$ be a random vector independent of $A$, taking its values in $\R^d_+$.

Let $A_1\krr A_N$ and $B_0$ be
independent random variables, where $N\ge2$ is a fixed natural number, $A_1\krr A_N$ are
independent copies of $A$, and $B_0$ is an independent copy of $B$.

The aim of this paper is to find a random vector $R\in\R^d_+$, independent of $A$ and $B$, which solves \Big(in law $\stackrel{\mathcal{D}}{=}$\Big) a generalized
multidimensional affine equation i.e.
\begin{align}\label{eq}
    R\stackrel{\mathcal{D}}{=}\sumi 1 N A_iR_i+B_0,
\end{align}
where $R_1\krr R_N$ are independent copies of $R\in\R^d_+$ and
independent of $A, A_1\krr A_N, B, B_0$, (see Theorem \ref{thm1} stated
below).

Furthermore, we would like to find possibly mild conditions, which
allows us to establish an asymptotic tail formula for $R$. More
precisely, we are interested in the existence $\c>0$, such that
\begin{align}\label{Rtail}
    \P(\{\iss R u>t\})\asymp t^{-\c},\ \ \ \mbox{as\ \ \ \ $t\to\8$},
\end{align}
for every $u\in\Sp$, (see Theorem \ref{thm2} stated below).

The one dimensional version of the equation \eqref{eq} has been
considered recently by Jelenkovi\'{c} and Olvera--Cravioto
\cite{gaJel1}, \cite{gaJel2} and \cite{gaJel3} in the context of Google's
PageRank algorithm. The authors solved equation \eqref{eq} and
justified formula \eqref{Rtail} using the renewal theorem. It is
worth to emphasize that the one dimensional version of equation
\eqref{eq} with $B=0$, was studied by Liu in a series of articles
(see for instance \cite{gaL2} and the references given there).

We are also motivated by the recent results of Buraczewski, Damek
and Guivarc'h \cite{gaBDG}, where the authors considered the
multidimensional version of equation \eqref{eq} with $B=0$, and
established formula \eqref{Rtail} with the help of Kesten's
renewal theorem \cite{K1} and the spectral method developed by
Guivarc'h and Le Page \cite{gaGL} and \cite{gaGL1}. Their approach sheds some new
light on multidimensional problems and fits perfectly to our
situation.
\medskip

In order to avoid repetitions in the sequel, and shorten article we have decided to state all
necessary definitions and notations in the introduction, and formulate our main results as general
as it is possible.

Let $M^1(G)$ denotes the set of all probability measures on $G$
endowed with the weak topology. We denote by $\supp\mu$ the support of the measure $\mu\in
M^1(G)$. If $E\subseteq G$, let $[E]$ be the subsemigroup of $G$
generated by the set $E$. For $n\in\N$ let
$S_n=A_n\cdot\ldots\cdot A_1\in G$, where $A_1, A_2,\ldots \in G$
is a sequence of independent copies of $G$--valued random matrix
$A$ distributed according to $\mu$.

A subsemigroup $[\supp\mu]$ of $G$ is called \textbf{contractive}
if $[\supp\mu]\cap G^{\circ}\not=\emptyset$. In other words, 
\begin{align}\label{con}
    \P\l(\bigcup_{n\in\N}\{S_n\in G^{\circ}\}\r)>0.
\end{align}
The condition \eqref{con} was considered by Hennion \cite{gaH}, Hennion and Herv\'{e} \cite{gaHH}
in the context of limit theorems for the products of positive random matrices.

An element $a\in\Gl(\Rd)$ is \textbf{proximal} if there exists a unique eigenvalue $\la_a$ (the
dominant eigenvalue) of $a$, such that $r(a)=\limn\|a^n\|^{1/n}=|\la_a|.$

According to the Perron--Frobenius theorem \cite{gaHJ} every $a\in G^{\circ}$ is proximal.
Moreover, for every $a\in G^{\circ}$ and its adjoint $a^*\in G^{\circ}$ it is possible to choose
$v_a, w_a\in\R^d_+$ such that $ v_a>0, w_a>0$ and
$$av_a=\la_av_a, \ \ \ a^*w_a=\la_aw_a, \ \ \ \iss {v_a}{w_a}=1, \ \ \ |w_a|=1.$$
The eigenvector $v_a$ determined by these relations will be called the dominant eigenvector of
$a\in G^{\circ}$. This means that we can write $\R^d=\R\ra v_a\oplus v_a^{\bot},$ and the spectral
radius of $a$, restricted to $v_a^{\bot}=\{x\in\Rd: \iss x{v_a}=0\}$ is strictly less than
$|\la_a|$. Furthermore, by the preceding relations we have
\begin{align}\label{frob1}
    \limn \frac{a^n}{r(a)^n}=v_a\otimes w_a,
\end{align}
where $v_a\otimes w_a$ is the matrix projector on $\R\ra v_a$. Since  $v_a\otimes w_ax=\iss
x{w_a}v_a$ for every $x\in\R^d$,  \eqref{frob1} immediately yields
\begin{align}\label{frob2}
    \limn a^n\ra x=\frac{v_a\otimes w_ax}{|v_a\otimes w_ax|}=\frac{v_a}{|v_a|}=\ov{v}_a\in\Sp,\ \ \ \mbox{for every $x\in\R^d$}.
\end{align}

A subsemigroup $\Gamma\subseteq\Gl(\Rd)$ is \textbf{strongly irreducible} if there does not exist a
finite number ($k\in\N$) of proper linear subspaces $V_1,\ldots, V_k$   of
  $\Rd$ such that
\begin{align}\label{si}
    \Gamma\l(\bigcup_{i=1}^kV_i\r)\subseteq\bigcup_{i=1}^kV_i.
\end{align}
If $E\subseteq\Gl(\Rd)$ we denote by $E^{prox}$ the set of all proximal elements of $E$. A
subsemigroup $\Gamma\subseteq\Gl(\Rd)$ is said to satisfy condition $(i-p)$ if $\Gamma$ is strongly
irreducible and $\Gamma^{prox}\not=\emptyset$. This condition was widely investigated by Guivarc'h
and Le Page \cite{gaGL} and \cite{gaGL1}, see also \cite{gaGR, gaGU, gaBDG} and the references given there.

A subsemigroup $[\supp\mu]\subseteq G$, where $\mu\in M^1(G)$, is
said to satisfy \textbf{condition} $(\mathcal{C})$ if $[\supp\mu]$
is contractive and strongly irreducible.
Clearly, condition $(\mathcal{C})$ implies condition $(i-p)$ with
$\Gamma=[\supp\mu]$.

For $s\ge 0$ we write
$$\ka(s)=\kappa_{\mu}(s) = \lim_{n\to\8}\l( \int_G\|a\|^s \mu^{*n}(da)\r)^{\frac 1n},$$
where $\mu^{*n}$ is the $n$--th convolution power of $\mu\in M^1(G)$. The  limit above exists and
it is equal to $\inf_{n\in\N}\l( \int_G\|a\|^s \mu^{*n}(da)\r)^{\frac 1n}$, because
$u_n(s)=\int_G\|a\|^s\mu^{*n}(da)$ is submultiplicative, i.e. $u_{m+n}(s)\le u_m(s)u_n(s)$ for
every $m, n\in\N$. Moreover,
$$ I_{\mu} = \l\{ s\in [0, \8): \kappa_{\mu}(s)<\8 \r\}
= \l\{ s\in[0, \8): \int_G\|a\|^s\mu(da)<\8 \r\}.
$$
Let $s_\8=\sup\l\{ s\ge0: \kappa_{\mu}(s)<\8 \r\} \in \R_+\cup\{\8\}$, then by the H\"older inequality  $I_{\mu}=[0,s_\8)$ or $I_{\mu}=[0,s_\8]$. For
technical reasons we have to assume that there is $s_1\leq
\frac{1}{2}$ such that $\E(\|A\|^{s_1})\le\frac 1N$.

Our ``existence'' result is the following
\begin{thm}\label{thm1}
Assume that $A$ is a $G$--valued random matrix distributed according to a
probability measure $\mu$ on $G$, and $B$ is a random vector independent of $A$, taking its values in $\R^d_+$, such that $\P(\{B>0\})>0$. Let $A_1\krr A_N$ and $B_0$ be
independent random variables as in \eqref{eq}, where $N\ge2$ is a fixed natural number, $A_1\krr A_N$ are
independent copies of $A$, and $B_0$ is an independent copy of $B$.
Suppose further that $[\supp\mu]\subseteq G$ satisfies
condition $(\mathcal{C})$ and there exist
$s_1\in(0, 1/2]$, and $s_2>s_1$ such that $\E(\|A\|^{s_1})\le\frac
1N$, $\E(\|A\|^{s_2})\le\frac 1N$, and $\E(|B|^{s_2})<\8$. Then
there exists a unique vector $R\in\R^d_+$ and its independent copies $R_1\krr R_N$
independent of $A, A_1\krr A_N, B, B_0$  which solve \eqref{eq}
in law. Moreover, $\E(|R|^s)<\8$ for every $s<s_2$.
\end{thm}
Section \ref{COS} contains a detailed proof of Theorem \ref{thm1}, which
is similar in spirit to that of \cite{gaJel1}. 
However, the multidimensional framework, we consider, provides
some 
difficulties which do not appear in the one dimensional case.
Namely, the method developed in \cite{gaJel1}, which gives
finiteness of appropriate moments for the solution of \eqref{eq},
breaks down in higher dimensions. This problem will be dealt with
the help of condition $(\mathcal{C})$.

\medskip
Let $\la_d$ be the Lebesgue measure on $\Rd$. If $\nu$ is a probability measure on $\Rd$, then by $\nu=\nu_a+\nu_s$ we denote its Lebesgue decomposition with respect to $\la_d$, where $\nu_a$ is the absolutely continuous part with respect to $\la_d$, i.e. $\nu_a\ll\la_d$, and $\nu_s$ is the singular part with respect to $\la_d$, i.e. $\nu_s\perp\la_d$. We have also $\nu_a\perp\nu_s$. Since $\nu$ is positive then its total variation $\|\nu\|=\nu(\Rd)=1$. We say that the measure $\nu$ is singular if $\|\nu_s\|=1$, otherwise $\nu$ is nonsingular, i.e. $\|\nu_s\|<1$.

Now we can state our main ``tail'' result.
\begin{thm}\label{thm2}
Fix a natural number $N\ge2$, a $G$--valued random matrix $A$
distributed according to $\mu$, and a random vector $B$ with law $\eta$, independent of $A$, taking its
values in $\R^d_+$, such that $\P(\{B>0\})>0$.
\begin{itemize}
\item Assume that $[\supp\mu]\subseteq G$ satisfies condition
$(\mathcal{C})$, and there is $s_1\in(0, 1/2]$, such that
$\E(\|A\|^{s_1})\le\frac 1N$. Moreover, we assume $s_{\8}>s_1$ and
$\lim_{s\to s_{\8}}\ka(s)>\frac 1N$. Then there exists $\c>s_1$
such that $N\ka(\c)=1$.

\item
Furthermore, if $\E(\|A\|^{\c}\log^+\|A\|)<\8$,
$E(|B|^{\c+\varepsilon})<\8$ for some $\eps>0$, and either
\begin{enumerate}
  \item [(i)] $\eta$ is nonsingular, i.e. $\|\eta_s\|<1$, or
  \item [(ii)] $\eta$ is singular, i.e. $\|\eta_s\|=1$, and $\P(\{\iss Bu=r\})=0$ for every $(u,
r)\in\Sp\times\R_+$.
\end{enumerate}
\end{itemize}
Then there
exists a positive function $e^{\c}:\Sp\mapsto(0, \8)$ and a
 constant $C_{\c}\ge0$ such that
\begin{align}\label{thm2a}
    \limx t{\8}t^{\c}\P(\{\iss Ru>t\})=C_{\c}e^{\c}(u)\ge0,
\end{align}
for every $u\in\Sp$, where $R\in\R^d_+$ is the stationary solution of the equation \eqref{eq} as in Theorem \ref{thm1}. Moreover, if $\c\ge1$ then $C_{\c}>0$, and the limit in \eqref{thm2a} is strictly positive.
\end{thm}
Now we give an example of singular measure $\eta$, i.e. $\|\eta_s\|=1$, on the plane $(d=2)$, such that $\eta(\{x\in\R^2: \iss xu=r\})=0$ for every $(u, r)\in\Sp\times\R_+$, and $\eta(\{x\in\R^2: x>0\})>0$. Define $S=\{(\cos \al, \sin \al):  0<\al<\pi/2\}\subseteq\Sp$ and let $\eta$ be the normalized one dimensional Lebesgue measure on $S$, i.e. $\supp\eta= \Sp$ and $\eta(S)=1$. It is not hard to see that $\eta$ is singular with respect to two dimensional Lebesgue measure $\la_2$. Obviously $\eta(\{x\in\R^2: x>0\})=\eta({S})=1,$ and notice that $\{x\in\R^2: \iss xu=r\}$ intersects ${S}$ at most two points, hence finally $\eta(\{x\in\R^2: \iss xu=r\})=0$.

\medskip

As we mentioned before, the proof is based on concepts of
\cite{gaBDG} with considerable complications determined by the
structure of equation \eqref{eq}. The most important tool which
allows us to establish relation \eqref{thm2a} is Kesten's renewal
theorem \cite{K1}.
We need to check that its assumptions are satisfied (see Section
\ref{KRT}). This is the most difficult part of the paper and requires the
spectral theory of transfer operators developed by Guivarc'h
and Le Page 
(\cite{gaGL}, \cite{gaGL1} and \cite{gaBDG}), which is summarized in Section \ref{TO}. But we  touch only a few aspects of their theory 
and restrict our attention to the results which will be used in
Sections \ref{COS} and \ref{KRT}. Guivarc'h and Le Page approach significantly
simplifies and clarifies proofs developed by Kesten in \cite{K},
and what is most important for us, it is applicable to our
situation.

\section*{Acknowledgements}
The results of this paper are part of the author's PhD thesis, written under the supervision of Ewa
Damek at the University of Wroclaw. I wish to thank her for many stimulating conversations and
several helpful suggestions during the preparation of this paper. I would like also to thank
Dariusz Buraczewski for beneficial discussions and comments.

\section{Transfer operators}\label{TO}
Let $\Cc(\Sp)$ be the space of continuous functions on $\Sp$ with
the supremum norm $|\ra|_{\8}$.
\newline $\H_{\eps}=\{\phi\in\Cc(\Sp):
\|\phi\|_{\eps}=|\phi|_{\8}+[\phi]_{\eps}<\8\}$, $\eps\in(0, 1]$
is the space of all $\eps$--H\"older functions on $\Sp$ with
$$
[\phi]_{\eps} = \sup_{x\not= y}\frac{|\phi(x)-\phi(y)|}{|x-y|^{\eps}}.
$$

Given a closed subset $V$ of $\Sp$, $M^1(V)$ denotes the set of
all probability measures on $V$, endowed with the weak topology.
We say that $U\subseteq\Sp$ is a subspace of $\Sp$, if
$U=V\cap\Sp$ for some subspace $V\subseteq\Rd$. A measure $\nu\in
M^1(\Sp)$ is said to be proper if $\nu(U)=0$ for every subspace
$U\varsubsetneqq\Sp$. Here and subsequently, $\La(\Gamma)=\{\ov
{v}_a\in\Sp: v_a\mbox{ is the dominant eigenvector of }
a\in\Gamma^{prox}\}$, where $\Gamma$ is a subsemigroup of $G$ such
that $\Gamma^{prox}\not=\emptyset$.

The following Proposition \ref{gr} due to Guivarc'h and Raugi
\cite{gaGR} (see also \cite{gaGU}) contains the relevant
properties of $(i-p)$ semigroups which will be used in the sequel.

\begin{prop}\label{gr}
Let $\mu\in M^1(G)$ and $\Gamma=[\supp\mu]$ satisfies condition $(i-p)$. Then there exists a unique
proper $\mu$--stationary measure $\nu\in M^1(\Sp)$ such that $\supp\nu=\Lambda(\Gamma)$.
Furthermore, $\Lambda(\Gamma)$ is the unique $\Gamma$ -- minimal subset of $\Sp$ (i.e. if $Z\subseteq\Sp$ is closed and $\Gamma\cdot Z\subseteq Z$, then $\Lambda(\Gamma)\subseteq Z$), and the subgroup of $\R^*_+$ generated by the set $\{|\la_a|:\ a\in \Gamma^{\rm prox}\}$ is dense in $\R^*_+$.
\end{prop}
Let $\mu\in M^1(G)$. For $s\in I_{\mu}$, $x\in\Sp$ and a
measurable function $\phi$ on $\Sp$ we consider the following
transfer operators
\begin{equation}
\label{ps}
\begin{split}
P^s\phi(x) &= \int_G |ax|^s \phi(a\cdot x) \mu(da),\\
P^s_*\phi(x) &= \int_G |a^*x|^s \phi(a^*\cdot x) \mu(da)=\int_G |ax|^s \phi(a\cdot x) \mu_*(da),
\end{split}
\end{equation}
where $\mu_*\in M^1(G)$ and $\mu_*(U) =\mu(\{a\in G: a^*\in U\})$ for every $U\in\mathcal{B}
or(G)$.\\

The main purpose of this section is to summarize a number of
properties of operators $P^s, P^s_*$, see
Theorem \ref{thmGL} below.

\begin{thm}\label{thmGL}
Assume that $\mu\in M^1(G)$, $s\in I_{\mu}$ and
$\Gamma=[\supp\mu]$ satisfies condition $(i-p)$. Then
\begin{itemize}
  \item there exists a unique probability measure $\nu^s\in M^1(\Sp)$, $(\nu^s_*\in M^1(\Sp))$ such that
\begin{itemize}
  \item[(i)] $P^s\nu^s = \ka(s)\nu^s$, $(P^s_*\nu_*^s = \ka(s)\nu^s_*)$.
  \item[(ii)]
  $\supp\nu^s=\La([\supp\mu])$, $(\supp\nu^s_*=\La([\supp\mu_*]))$ and it is not contained in any proper subspace of $\Sp$.
  \item[(iii)] $I_{\mu}\ni s\mapsto\nu^s\in M^1(\Sp)$, $(I_{\mu}\ni s\mapsto\nu^s_*\in M^1(\Sp))$ is continuous in the weak topology.
\end{itemize}
\item $I_{\mu}\ni s\mapsto\ka(s)$ is strictly $\log$--convex function.
  \item  there exists a
unique $\underline{s}$--H\"{o}lder continuous function $e^s:\Sp\mapsto(0, \8)$ with
$\underline{s}=\min\{s, 1\}$ such that
\begin{enumerate}
  \item [(i)] $P^se^s=\ka(s)e^s$.
  \item [(ii)] $e^s$ is given by the formula
\begin{align*}
    e^s(x) = \int_{\Sp}  \iss xy^s \nu^s_*(dy),\ \ \ \mbox{for $x\in\Sp$}.
\end{align*}
  \item [(iii)] $I_{\mu}\ni s\mapsto e^s\in \Cc(\Sp)$ is continuous in the uniform topology.
\end{enumerate}
\end{itemize}
\begin{itemize}
  \item Moreover, there exists a unique stationary measure $\pi^s
\in M^1(\Sp)$, $(\pi^s_*\in M^1(\Sp))$ for operator $Q^sf=\frac {P^s(e^sf)}{\kappa(s)e^s}$,
$\l(Q^s_*f=\frac {P^s_*(e^sf)}{\kappa(s)e^s}\r)$ where $f\in\Cc(\Sp)$, such that
\begin{itemize}
  \item[(i)] $\pi^s=\frac{e^s\nu^s}{\nu^s(e^s)}$, $\l(\pi^s_*=\frac{e^s\nu^s_*}{\nu^s_*(e^s)}\r)$.
  \item[(ii)] $(Q^s)^nf$, $((Q^s_*)^nf)$ converges uniformly to $\pi^s(f)$, $(\pi^s_*(f))$ for any $f\in
  \Cc(\Sp)$.
  \item[(iii)] $\supp\pi^s=\La([\supp\mu])$, $(\supp\pi^s_*=\La([\supp\mu_*]))$.
\end{itemize}
\end{itemize}
\end{thm}
This result was proved by Guivarc'h and Le Page
and its, quite long and far from
being obvious, proof can be found in \cite{gaGL} and \cite{gaGL1}.
Notice that in view of the cocycle property $\sigma^s(x,a_2
a_1)=\sigma^s(x, a_1)\sigma^s(a_1\cdot x,a_2)$, $a_1, a_2\in G$,
$x\in\Sp$ of
\begin{equation}
\label{sigma} \sigma^s(x,a) = |ax|^s\frac{e^s(a\cdot x)}{e^s(x)},\ \ \ 
\end{equation}
the Markov operators $Q^s$ and $Q^s_*$ defined in Theorem \ref{thmGL} can be rewritten in the following form
\begin{align}
\label{Qsn} (Q^s)^n \phi(x)& = \int_{G}\phi
(a\cdot x) q^s_n(x,a)\mu^{*n}(da),\\
\label{Qs1n} (Q^s_*)^n \phi(x)& =\int_{G}\phi (a\cdot x) q^s_n(x,a)\mu^{*n}_*(da),
\end{align}
where
\begin{equation}
\label{qsn} q^s_n(x,a) = \frac 1{\kappa^n(s)}\frac{e^s(a\cdot
x)}{e^s(x)}|ax|^s=\frac{\sigma^s(x,a)}{\ka^n(s)},
\end{equation}
$n\in\N$, $x\in\Sp$, $a\in G$ and $\phi$ is an arbitrary measurable function on $\Sp$.

\section{Construction of the solution}\label{COS}
Recall that $A$ stands for a $G$--valued random matrix distributed
according to the measure $\mu\in M^1(G)$, and $B$ for a random
vector taking its values in $\R_+^d$, independent of $A$. In this section we
construct a solution of the equation \eqref{eq}. The idea of the
construction goes back to \cite{gaJel1}. It is not difficult to
imagine that we have to study a sequence of random variables that
are obtained by iterating \eqref{eq}. Let $N\ge2$ be a fixed natural number and $R_{0, 1}^*\krr R_{0,
N}^*$ be independent and identically distributed (i.i.d.) copies
of the initial random variable $R_{0}^*\in\R^d_+$. We consider the
sequence $(R^*_n)_{n\ge0}$ such that
\begin{align}\label{eq1}
R_{n+1}^*=\sumk 1 N A_{n+1, k}R_{n, k}^*+B_{n+1},\ \ \ \mbox{for every $n\ge0$},
\end{align}
where $A_{n+1, 1}\krr A_{n+1, N}$, $B_{n+1}$ and $R_{n, 1}^*\krr R_{n, N}^*$, $n\ge 0$ are independent. Moreover, for $n\ge 1$ $R_{n, 1}^*\krr R_{n, N}^*$  are i.i.d. copies of
$R_{n}^*$ from the previous iteration. For $n\ge0$ $A_{n+1, 1}\krr A_{n+1, N}$
are i.i.d. copies of $A$ and $B_{n+1}$ is an independent copy of $B$.

We will look more closely at the sequence $(R^*_n)_{n\ge0}$. Let
 $\mathcal{A}=\{A^{}_{i_1\krr i_n}: (i_1\krr
i_n)\in\{1\krr N\}^n,\ n\in\N\}$ be the set consisting of i.i.d.
copies of $A$, and  $\mathcal{B}=\{B^{}_{i_1\krr i_n}: (i_1\krr
i_n)\in\{1\krr N\}^n,\ n\in\N\}\cup\{B_0\}$ the set  consisting of i.i.d.
copies of $B$ independent of $\mathcal{A}$. Additionally we assume
that $A_0=\Id$ a.s. and the initial random variable $R_0^*$ is always independent of $A$, $B$, $\mathcal{A}$ and $\mathcal{B}$.

Now let $W_0=A_0B_0=B_0$ a.s.,
\begin{align}\label{defW}
    W_{n}=\sum_{(i_1\krr i_n)\in\{1\krr N\}^n}A_{i_1}A_{i_1, i_2}\raa A_{i_1\krr i_n}B_{i_1\krr
    i_n}, \ n\ge1 ,
\end{align}
and for $n\ge 0$
 \begin{align}\label{defRn}
   R^{(n)}=\sumi 0 n W_i
 \end{align}
 be the partial sum of the sequence $(W_n)_{n\ge0}$. Then
\begin{align}\label{defR}
    R=\limn R^{(n)}=\sumi 0 {\8}W_i \mbox{ a.s.}
\end{align}
is a candidate for a solution of \eqref{eq}. It is not hard to see
that $W_n$ satisfies
\begin{align}\label{defW1}
    W_n&=\sum_{(i_1\krr i_n)\in\{1\krr N\}^n}A_{i_1}A_{i_1, i_2}\raa A_{i_1\krr i_n}B_{i_1\krr
    i_n}\\
  \nonumber  &=\sumk 1 N A_{k}\l(\sum_{(k, i_2\krr i_n)\in\{1\krr N\}^n}A_{k, i_2}\raa A_{k, i_2\krr i_n}B_{k, i_2\krr
    i_n}\r)=\sumk 1 N A_{k}W_{n-1, k},
\end{align}
where $A_k$ and $W_{n-1, k}$ are independent of each other and
$W_{n-1, 1}\krr W_{n-1, N}$ have the same distribution as
$W_{n-1}$. In view of the above calculations, $R^{(n)}$ satisfies
the recursion
\begin{align}
   \label{eqn} R^{(n)}=\sumk 1 N A_kR_k^{(n-1)}+B_0,
\end{align}
for every $n\in\N$, where $R^{(n-1)}_1\krr R^{(n-1)}_N$ are independent copies of $R^{(n-1)}$.

To obtain a solution with an initial condition, let $R^*_{0,
(i_1\krr i_n)}$, $(i_1\krr i_n)\in\{1\krr N\}^n$, $n\in\N$, be
i.i.d. copies of the initial random variable $R_{0}^*\in\R^d_+$
independent of the families $\mathcal{A}$ and $\mathcal{B}$. For
$n\ge1$, we define similarly as in \eqref{defW}
\begin{align}\label{defWR}
    W_{n}(R_{0}^*)=\sum_{(i_1\krr i_n)\in\{1\krr N\}^n}A_{i_1}A_{i_1, i_2}\raa A_{i_1\krr i_n}R^*_{0, (i_1\krr
    i_n)}.
\end{align}
Moreover, as in \eqref{defW1}, we obtain
$W_{n}(R_{0}^*)\stackrel{\mathcal{D}}{=}\sumk 1 N A_{k}W_{n-1,
k}(R_{0}^*)$, where $A_k$ and $W_{n-1, k}(R_{0}^*)$ are
independent of each other and $W_{n-1, 1}(R_{0}^*)\krr W_{n-1,
N}(R_{0}^*)$ have the same distribution as $W_{n-1}(R_{0}^*)$. Now
we have following
\begin{lem}
Assume now that $(R^*_n)_{n\ge0}$ and $(R^{(n)})_{n\ge0}$ are the
sequences defined in \eqref{eq1} and \eqref{defRn} respectively,
then for every $n\in\N$ we have
\begin{align}\label{WRW}
R^*_n\stackrel{\mathcal{D}}{=}R^{(n-1)}+W_{n}(R_{0}^*).
\end{align}
\end{lem}
\begin{proof}
Observe that for $n=1$, \eqref{WRW} follows from definition. For more details we refer to \cite{gaJel1}.
\end{proof}
Now we have simple, but very useful
\begin{lem}\label{l2}
Under the assumptions of Theorem \ref{thmGL} there exists $c_s>0$ such that for
every $n\in\N$ we have
\begin{align}\label{l2a}
    c_s \int_{G}\|a\|^s\mu^{n}(da)\le \ka^n(s)\le\int_{G}\|a\|^s\mu^{n}(da).
\end{align}
\end{lem}
\begin{proof}
We refer to \cite{gaGL}.
\end{proof}

To take the limit in \eqref{defR} we need an estimate for
$\E\l(|W_n|^{s}\r)$.
Suppose for a moment that $s\leq 1$. Then, in view of inequality
\eqref{l2a}, we have
\begin{align*}
    \E\l(|W_n|^{s}\r)&\le \E\l(\sum_{(i_1\krr i_n)\in\{1\krr N\}^n}\|A_{i_1}A_{i_1, i_2}\raa A_{i_1\krr i_n}\|^{s}|B_{i_1\krr
    i_n}|^{s}\r)\\
    &\le\sum_{(i_1\krr i_n)\in\{1\krr N\}^n}\E\l(\|A_{i_1}A_{i_1, i_2}\raa A_{i_1\krr
    i_n}\|^{s}\r)\E\l(|B|^{s}\r)\\
    &=N^n\int_{G}\|a\|^{s}\mu^{*n}(da)\E\l(|B|^{s}\r)\le
    \frac{1}{c_s}\E\l(|B|^{s}\r)N^n\ka^n(s).
\end{align*}
We would like to show that for an appropriate $s>0$, not necessarily
less or equal $1$, $\E\l(|W_n|^{s}\r)$ decays exponentially. This
is contained in  Lemma \ref{mom2lem}. For the sake of computations
we have to assume that there exist $s_1\in(0, 1/2]$ such that
$\E(\|A\|^{s_1})\le\frac 1N$.

\begin{lem}\label{mom2lem}
Assume that $[\supp\mu]\subseteq G$ satisfies condition
$(\mathcal{C})$, and there exist $s_1\in(0, 1/2]$, and $s_2>1$
such that $\E(\|A\|^{s_1})\le\frac 1N$, $\E(\|A\|^{s_2})\le\frac
1N$, and $\E(|B|^{s_2})<\8$. Then for every $s\in(s_1, s_2)$,
there exist finite constants $K_s>0$ and $\eta<1$ such that for
every $n\in\N$
\begin{align}\label{mom2}
    \E\l(|W_n|^{s}\r)\le K_s\eta^n.
\end{align}
\end{lem}
\begin{proof}
By Theorem \ref{thmGL} $\ka (s)$ is strictly $\log$--convex so
$N\ka(s)<1$, for every $s\in(s_1, s_2)$ and for $s\leq 1$,
\eqref{mom2} follows from the calculation above. From now we
assume that $s\in(1, s_2)$ and it is fixed.
Let  $S_{i_1\krr i_n}=A_{i_1}A_{i_1, i_2}\raa A_{i_1\krr
i_n}$ for $(i_1\krr i_n)\in\{1\krr N\}^n$ and $n\in\N$. We order
the set of indices writing $\{1\krr N\}^n=\{{\mathbf{i}_1}\krr
{\mathbf{i}_{N^n}}\}$ and we choose $p\in\N$ and $p\ge2$, such
that $p-1<s\le p$. Then $s_1\le1/2<s/p\le1$ and
\begin{align*}
    \E\l(|W_n|^{s}\r)&\le \E\l(\l(\sum_{(i_1\krr i_n)\in\{1\krr N\}^n}|S_{i_1\krr i_n}B_{i_1\krr
    i_n}|^{s/p}\r)^p\r)\\
    &=\E\l(\sum_{j_{\mathbf{i}_1}+\ldots+j_{\mathbf{i}_{N^n}}=p}\binom{p}{j_{\mathbf{i}_1},\ldots,j_{\mathbf{i}_{N^n}}}|S_{\mathbf{i}_1}B_{\mathbf{i}_1}|^{sj_{\mathbf{i}_1}/p}
    \raa |S_{\mathbf{i}_{N^n}}B_{\mathbf{i}_{N^n}}|^{sj_{\mathbf{i}_{N^n}}/p}\r)\\
    &\le
    \sum_{j_{\mathbf{i}_1}+\ldots+j_{\mathbf{i}_{N^n}}=p}\binom{p}{j_{\mathbf{i}_1},\ldots,j_{\mathbf{i}_{N^n}}}\E\l(\l(\|S_{\mathbf{i}_1}\||B_{\mathbf{i}_1}|\r)^{sj_{\mathbf{i}_1}/p}\r)
    \raa \E\l(\l(\|S_{\mathbf{i}_{N^n}}\||B_{\mathbf{i}_{N^n}}|\r)^{sj_{\mathbf{i}_{N^n}}/p}\r).
\end{align*}
Notice that
$\E\l(|B_{\mathbf{i}_1}|^{sj_{\mathbf{i}_1}/p}\r)\raa\E\l(|B_{\mathbf{i}_{N^n}}|^{sj_{\mathbf{i}_{N^n}}/p}\r)
=\|B\|_{sj_{\mathbf{i}_1}/p}^{sj_{\mathbf{i}_1}/p}\raa\|B\|_{sj_{\mathbf{i}_{N^n}}/p}^{sj_{\mathbf{i}_{N^n}}/p}\le\|B\|_s^s,$
since $\|B\|_{r}=\E(|B|^r)^{1/r}$ is increasing and $\|B\|_0=1$.
This implies that
\begin{align*}
    &\sum_{j_{\mathbf{i}_1}+\ldots+j_{\mathbf{i}_{N^n}}=p}\binom{p}{j_{\mathbf{i}_1},\ldots,j_{\mathbf{i}_{N^n}}}\E\l(\l(\|S_{\mathbf{i}_1}\||B_{\mathbf{i}_1}|\r)^{sj_{\mathbf{i}_1}/p}\r)
    \raa \E\l(\l(\|S_{\mathbf{i}_{N^n}}\||B_{\mathbf{i}_{N^n}}|\r)^{sj_{\mathbf{i}_{N^n}}/p}\r)\\
    &\le\E\l(|B|^s\r)\sum_{j_{\mathbf{i}_1}+\ldots+j_{\mathbf{i}_{N^n}}=p}\binom{p}{j_{\mathbf{i}_1},\ldots,j_{\mathbf{i}_{N^n}}}\E\l(\|S_{\mathbf{i}_1}\|^{sj_{\mathbf{i}_1}/p}\r)
    \raa \E\l(\|S_{\mathbf{i}_{N^n}}\|^{sj_{\mathbf{i}_{N^n}}/p}\r)\\
    &=\E\l(|B|^s\r)\sum_{j_{\mathbf{i}_1}+\ldots+j_{\mathbf{i}_{N^n}}=p}\binom{p}{j_{\mathbf{i}_1},\ldots,j_{\mathbf{i}_{N^n}}}\int_G\|a\|^{sj_{\mathbf{i}_1}/p}\mu^{*n}(da)
    \raa \int_G\|a\|^{sj_{\mathbf{i}_{N^n}}/p}\mu^{*n}(da).
\end{align*}
Observe that by the inequality \eqref{l2a}, there exist constants $c_{sj_{\mathbf{i}_1}/p},
c_{sj_{\mathbf{i}_2}/p}\krr c_{sj_{\mathbf{i}_{N^n}}/p}\in(0,
    1]$, such that for all $n\in\N$
\begin{align*}
\int_G\|a\|^{sj_{\mathbf{i}_1}/p}\mu^{*n}(da)&\le
    c_{sj_{\mathbf{i}_1}/p}^{-1}\ka^n(sj_{\mathbf{i}_1}/p),\\
    \int_G\|a\|^{sj_{\mathbf{i}_2}/p}\mu^{*n}(da)&\le
    c_{sj_{\mathbf{i}_2}/p}^{-1}\ka^n(sj_{\mathbf{i}_2}/p),\\
    &\vdots\\
    \int_G\|a\|^{sj_{\mathbf{i}_{N^n}}/p}\mu^{*n}(da)&\le
    c_{sj_{\mathbf{i}_{N^n}}/p}^{-1}\ka^n(sj_{\mathbf{i}_{N^n}}/p).
\end{align*}
Since $j_{\mathbf{i}_1}, j_{\mathbf{i}_2}\krr
j_{\mathbf{i}_{N^n}}\in\{0,1\krr p\}$, the constants above do not
depend on $n\in\N$ and we may define
$c_{p,s}=\max\{c_0^{-1},c_{s/p}^{-1}, c_{2s/p}^{-1}\krr
c_{(p-1)s/p}^{-1}, c_s^{-1}\}$ that dominates all of them.

When $N^n\le p$, we have
   \begin{align}
\label{mom2a}\int_G\|a\|^{sj_{\mathbf{i}_1}/p}\mu^{*n}(da)
    \raa \int_G\|a\|^{sj_{\mathbf{i}_{N^n}}/p}\mu^{*n}(da)\le c_{p,s}^p\ka^n(sj_{\mathbf{i}_1}/p)\raa \ka^n(sj_{\mathbf{i}_{N^n}}/p).
    \end{align}
Therefore,
\begin{align*}
    \sum_{j_{\mathbf{i}_1}+\ldots+j_{\mathbf{i}_{N^n}}=p}&\binom{p}{j_{\mathbf{i}_1},\ldots,j_{\mathbf{i}_{N^n}}}\int_G\|a\|^{sj_{\mathbf{i}_1}/p}\mu^{*n}(da)
    \raa \int_G\|a\|^{sj_{\mathbf{i}_{N^n}}/p}\mu^{*n}(da)\\
    &\le
    c_{p,s}^p\sum_{j_{\mathbf{i}_1}+\ldots+j_{\mathbf{i}_{N^n}}=p}\binom{p}{j_{\mathbf{i}_1},\ldots,j_{\mathbf{i}_{N^n}}}\ka^n(sj_{\mathbf{i}_1}/p)\raa
    \ka^n(sj_{\mathbf{i}_{N^n}}/p)\\
    &\le c_{p,s}^p\cdot\max\{\ka(s/p), \ka(2s/p)\krr\ka((p-1)s/p), \ka(s)\}^n\cdot\sum_{j_{\mathbf{i}_1}+\ldots+j_{\mathbf{i}_{N^n}}=p}\binom{p}{j_{\mathbf{i}_1},\ldots,j_{\mathbf{i}_{N^n}}}\\
    &\le c_{p,s}^pN^{pn}\cdot\max\{\ka(s/p), \ka(2s/p)\krr\ka((p-1)s/p), \ka(s)\}^n\\
    & \le c_{p,s}^p p^{p-1}N^{n}\cdot\max\{\ka(s/p), \ka(2s/p)\krr\ka((p-1)s/p), \ka(s)\}^n,
\end{align*}
since $ks/p\in(s_1, s_2)$ for every $k\in\{1\krr p\}$.
This yields \eqref{mom2} with $K_s=c_{p,s}^p
p^{p-1}\E\l(|B|^s\r)<\8$ and
$\eta=N\cdot\max\{\ka(s/p)\krr \ka(s)\}<1$. Here the
assumption $s_1\leq 1/2$ is indispensable, because it guarantees that
$N\cdot \ka(s/p)<1$.

When $N^n> p$, 
\eqref{mom2a} also holds with the universal constant
$c_{p,s}^p$ which does not depend on $n\in\N$, 
 but we have to estimate
\begin{align*}
    \sum_{j_{\mathbf{i}_1}+\ldots+j_{\mathbf{i}_{N^n}}=p}\binom{p}{j_{\mathbf{i}_1},\ldots,j_{\mathbf{i}_{N^n}}}\ka^n(sj_{\mathbf{i}_1}/p)\raa
    \ka^n(sj_{\mathbf{i}_{N^n}}/p),
\end{align*}
in a more subtle way. Before we do that we need to introduce a
portion of necessary definitions.

 For every $r\le k$, and $j_{1}\le\ldots\le j_{k}$, let
\begin{align*}
    L(j_{1},\ldots, j_{k})=\binom{k}{l_1,l_2\krr l_r},
\end{align*}
when $j_{1}=\ldots=j_{l_1}< j_{{l_1}+1}=\ldots=j_{l_2+l_1}<
j_{{l_2+l_1}+1}=\ldots=j_{l_3+l_2+l_1}<\ldots<j_{{l_{r-1}+\ldots+l_1}+1}=\ldots=j_{l_r+\ldots+l_1}$
and $l_1+l_2+\ldots+l_r=k$. Then it is not difficult to see that for every $k\le p$
\begin{align*}
    L(j_{1},\ldots, j_{k})&\le k!,\\
    \binom{p}{j_{1},\ldots, j_{k}}&\le p!,\\
    \binom{N^n}{k}L(j_{1},\ldots, j_{k})&\le \frac{N^n!}{(N^n-k)!}\le N^{kn}.
\end{align*}
Let now $\eta=\max\{\eta_1, \eta_2\krr \eta_p\}<1$, where
$$\eta_k=\max\{(N\ka(sj_{1}/p))\raa(N\ka(sj_{k}/p)): j_{1}+\ldots+j_{k}=p,\mbox{ and } j_{1}\le
\ldots\le j_{k}\}<1.$$ This implies that
\begin{align*}
    &\sum_{j_{\mathbf{i}_1}+\ldots+j_{\mathbf{i}_{N^n}}=p}\binom{p}{j_{\mathbf{i}_1},\ldots,j_{\mathbf{i}_{N^n}}}\ka^n(sj_{\mathbf{i}_1}/p)\raa
    \ka^n(sj_{\mathbf{i}_{N^n}}/p)=N^n\ka^n(s)\\
    &+\binom{N^n}{2}\sum_{\genfrac{}{}{0pt}{}{j_1+j_2=p}{j_1\le
j_2\ \ \ }}
    \binom{p}{j_1, j_2}L(j_{1}, j_{2})\ka^n(sj_{1}/p)
    \ka^n(sj_{2}/p)\\
    &+\binom{N^n}{3}\sum_{\genfrac{}{}{0pt}{}{j_{1}+j_{2}+j_{3}=p}{j_{1}\le
j_{2}\le j_{3}\ \ \ }}
    \binom{p}{j_{1}, j_{2}, j_{3}}L(j_{1}, j_{2}, j_{3})\ka^n(sj_{1}/p)
    \ka^n(sj_{{2}}/p)\ka^n(sj_{{3}}/p)\\
    &\vdots\\
    &+\binom{N^n}{k}\sum_{\genfrac{}{}{0pt}{}{j_{1}+\ldots+j_{k}=p}{j_{1}\le
\ldots\le j_{k}\ \ \ }}
    \binom{p}{j_{1},\ldots, j_{k}}L(j_{1},\ldots, j_{k})\ka^n(sj_{1}/p)
    \raa\ka^n(sj_{{k}}/p)\\
    &\vdots\\
    &+\binom{N^n}{p}\sum_{\genfrac{}{}{0pt}{}{j_{1}+\ldots+j_{{p}}=p}{j_{1}\le
\ldots\le j_{{p}}\ \ \ }}
    \binom{p}{j_{1},\ldots, j_{{p}}}L(j_{1},\ldots, j_{{p}})\ka^n(sj_{1}/p)
    \raa\ka^n(sj_{{p}}/p)\\
    &\le N^n\ka^n(s)+\sum_{\genfrac{}{}{0pt}{}{j_1+j_2=p}{j_1\le
j_2\ \ \ }}p!N^{2n}\ka^n(sj_{1}/p)
    \ka^n(sj_{2}/p)\\
    &+\ldots+\sum_{\genfrac{}{}{0pt}{}{j_{1}+\ldots+j_{k}=p}{j_{1}\le
\ldots\le j_{k}\ \ \ }}p!N^{kn}\ka^n(sj_{1}/p)
    \raa\ka^n(sj_{{p}}/p)\\
    &+\ldots+\sum_{\genfrac{}{}{0pt}{}{j_{1}+\ldots+j_{{p}}=p}{j_{1}\le
\ldots\le j_{{p}}\ \ \ }}p!N^{pn}\ka^n(sj_{1}/p)
    \raa\ka^n(sj_{{p}}/p)
    \le \eta^n\cdot p!\sumk 1 p\sum_{\genfrac{}{}{0pt}{}{j_{1}+\ldots+j_{k}=p}{j_{1}\le
\ldots\le j_{k}\ \ \ }}1\\
&\le\eta^n\cdot p!\sumk 1
p\binom{p-1}{k-1}\le2^{p-1}p!\cdot\eta^n.
\end{align*}
Hence in this case \eqref{mom2} follows with $K_s=2^{p-1}p!c_{p, s}^p\E\l(|B|^s\r)<\8$ and
$\eta<1$.
\end{proof}
\begin{proof}[Proof of Theorem \ref{thm1}] First of all we show that $\E(|R|^s)<\8$ for
every $s<s_2$.
By Lemma \ref{mom2lem} there exist $\eta<1$ and $K_s<\8$ such that for every $n\in\N$ we have $\E(|W_n|^s)\le K_s\eta^n$. Observe now
\begin{align*}
    \E(|R|^s)=\E(\liminf_{n\to\8}|R^{(n)}|^s)\le\liminf_{n\to\8}\E(|R^{(n)}|^s)\le\liminf_{n\to\8}\E\l(\sumk 0
    n|W_k|\r)^s.
\end{align*}
When $0<s\le1$, we have
\begin{align*}
   \liminf_{n\to\8}\E\l(\sumk 0
    n|W_k|\r)^s\le\liminf_{n\to\8}\E\l(\sumk 0
    n|W_k|^s\r)\le\liminf_{n\to\8}K_s\sumk 0 n \eta^k=\frac{K_s}{1-\eta}<\8.
\end{align*}
When $s>1$, we have
\begin{align*}
    \liminf_{n\to\8}\E\l(\sumk 0
    n|W_k|\r)^s&\le\liminf_{n\to\8}\l(\sumk 0
    n\E\l(|W_k|^s\r)^{1/s}\r)^s\\
    &\le\liminf_{n\to\8}K_s\l(\sumk 0
    n\eta^{k/s}\r)^s=\frac{K_s}{(1-\eta^{1/s})^s}<\8.
\end{align*}
It immediately implies that $\E(|R|^s)<\8$, which in turn gives $|R|<\8$ a.s..

Now we want to show that $R$ is the unique solution of \eqref{eq}.
It is enough to show that $R^*_n$, with arbitrary initial random
variable $R^*_0\in\R^d_+$ such that $\E(|R^*_0|^{r})<\8$ where
$r>0$, converge weakly to $R$ as $n\to\8$. We show that
$\E(f(R^*_n))\ _{\overrightarrow{n\to\8}}\ \E(f(R))$ for an
arbitrary uniformly continuous function $f$ defined on $\Rd$. Fix
$\eps>0$, and choose $\de>0$ such that
$$|x-y|<\de\ \Longrightarrow\ |f(x)-f(y)|<\eps.$$ By \eqref{WRW} we know that $R^*_n\stackrel{\mathcal{D}}{=}R^{(n-1)}+W_n(R^*_0)$ for every $n\in\N$, hence
\begin{align*}
    \l|\E(f(R^*_n)-f(R))\r|\le\l|\E(f(R^{(n-1)}+W_n(R^*_0))-f(R^{(n-1)}))\r|+\l|\E(f(R^{(n-1)})-f(R))\r|.
\end{align*}
It is enough to show that $\l|\E(f(R^{(n-1)}+W_n(R^*_0))-f(R^{(n-1)}))\r|\
_{\overrightarrow{n\to\8}}\ 0$. Fix $s<\min\{r, 1\}$ and observe that
\begin{align*}
    \l|\E(f(R^{(n-1)}+W_n(R^*_0))-f(R^{(n-1)}))\r|&\le\E(|\I{\{|W_n(R^*_0)|\le\de\}}(f(R^{(n-1)}+W_n(R^*_0))-f(R^{(n-1)}))|)\\
    &+\E(|\I{\{|W_n(R^*_0)|>\de\}}(f(R^{(n-1)}+W_n(R^*_0))-f(R^{(n-1)}))|)\\
    &\le\eps\P(\{|W_n(R^*_0)|\le\de\})+2M_f\P(\{|W_n(R^*_0)|>\de\})\\
    &\le\eps+2M_f\P(\{|W_n(R^*_0)|>\de\})\le\eps+2M_f\frac{\E(|W_n(R^*_0)|^s)}{\de^s}\\
    &\le\eps+\frac{2M_fK_s\E(|R^*_0|^s)}{\de^s}(N\ka(s))^n\ _{\overrightarrow{n\to\8}}\ \eps,
\end{align*}
since $\eps>0$ is arbitrary we have shown $\E(f(R^*_n))\ _{\overrightarrow{n\to\8}}\ \E(f(R))$, and
Theorem \ref{thm1} follows.
\end{proof}

\section{Application of Kesten's renewal theorem}\label{KRT}
In order to prove Theorem \ref{thm2}, as mentioned in the
introduction, we will use Kesten's renewal theorem \cite{K1} which
allows us to describe the desired tail asymptotic \eqref{thm2a}.
Before we state Kesten's theorem we have to introduce necessary
definitions and to prove a number of auxiliary results. They are
contained in the three lemmas of Section \ref{KRT1} and they will be
used later on 
to check that the assumptions
of Kesten's renewal theorem are satisfied in our settings. The
material presented in this section is adapted form \cite{gaBDG},
\cite{gaGL}, \cite{gaGL1}  and \cite{K}.
\subsection{Some general results}\label{KRT1}
At first we define the probability space $\O=G^{\N}$. $\mathcal{B}or(X)$ stands for the Borel
$\si$--field of the space $X$. For any sequence $\o=(a_1, a_2, \ldots)\in\O$ we write
\begin{align*}
    S_n(\o)=a_n\cdot\ldots\cdot a_1\in G,\ \ \mbox{ for $n\in\N$ and }\ S_0(\o)=\Id\in G.
\end{align*}
Let $\theta:\O\mapsto\O$ be the shift on $\O$, i.e.
\begin{align*}
    \te((a_1, a_2, \ldots))= (a_2, a_3, \ldots),\ \ \mbox{ for every $\o=(a_1, a_2, \ldots)\in\O$.}
\end{align*}
As in Section \ref{TO} (see \eqref{sigma} and \eqref{qsn}), for every $n\in\N$, we define the kernel
\begin{align*}
    q_n^s(x,\omega) = \prod_{k=1}^n q^s_1(S_{k-1}(\omega)\cdot x,a_k),\ \ \mbox{ for every $x\in\Sp$ and $\o=(a_1, a_2, \ldots)\in\O$.}
\end{align*}
The cocycle property gives a very useful relation, i.e. for every
$m, n\in\N$, $x\in\Sp$ and $\o\in\O$ we have
\begin{align}\label{coc}
    q_{m+n}^s(x,\omega)=q_{n}^s(x, S_n(\omega))q_{m}^s(S_n(\omega)\ra x, S_m(\te^n(\omega))).
\end{align}
The Kolmogorov's consistency theorem guarantees the existence of the probability measure $\Q_x^s$
on $\O$ being the unique extension of measures $q_k^s(x,g)\mu^{\ast k}(dg)$. Next we define the
probability measure
\begin{align*}
    \Q^s = \int_{\Sp} \Q_x^s\pi^s(dx),\ \ \mbox{ on }\ \ \O,
\end{align*}
where $\pi^s$ is the unique $Q^s$ stationary measure on $\Sp$ (see Theorem \ref{thmGL}). By
 $\E_x^s$ we denote the expectation corresponding to $\Q_x^s$. We extend the probability space $\O$ to $\go = \Sp\times \O$. Let $\gt:\go\mapsto\go$ be the
shift defined by
\begin{align*}
    \gt(x,\o) = (a_1\cdot x,\theta(\o)),\ \ \mbox{ for every $x\in\Sp$ and  $\o=(a_1, a_2, \ldots)\in\O$.}
\end{align*}
We now define the probability measure $\gq$ on $\go$ as follows
\begin{align*}
    \gq = \int_{\Sp} \delta_x \otimes
    \Q_x^s\pi^s(dx).
\end{align*}
In the same way, starting with $\mu_*$  instead of $\mu$, we
define the measure $\Q_x^{s,*}$, and $\E_x^{s,*}$ denotes its
expectation. Moreover, the probabilities $\Q^{s,*}$ and
$^aQ^{s,*}$ are defined similarly, i.e.
\begin{align*}
\Q^{s, *} = \int_{\Sp} \Q_x^{s, *}\pi^s_*(dx),\ \ \mbox{ and }\ \
    ^aQ^{s,*} = \int_{\Sp} \delta_x \otimes
    \Q_x^{s, *}\pi^s_*(dx),
\end{align*}
where $\pi^s_*$ is the unique $Q^s_*$ stationary measure on $\Sp$
(see Theorem \ref{thmGL}). Let $\o^*=(a_1^*, a_2^*,\ldots)\in\O$
for every $\o=(a_1, a_2, \ldots)\in\O$. Then
$S_n(\o^*)=a_n^*\cdot\ldots\cdot a_1^*\in G$.
\begin{rem}\label{remerg}
The properties of the stationary measures $\pi^s$ and $\pi^s_*$
developed in Section \ref{TO} imply that $(\O, \mathcal{B}or(\O), \Q^{s}, \theta)$, $(\O, \mathcal{B}or(\O),
\Q^{s,*}, \theta)$, $(\go, \mathcal{B}or(\go), {^aQ^{s}}, \gt)$ and $(\go, \mathcal{B}or(\go),
{^aQ^{s,*}}, \gt)$ are ergodic.
\end{rem}

From now we will work with the measures $\Q_x^{s,*}$, $\pi^s_*$,
$\Q^{s, *}$ and $^aQ^{s,*}$. Clearly, all the results stated below
remain valid for the measures $\Q_x^{s}$, $\pi^s$, $\Q^{s}$ and
$^aQ^{s}$.

We begin with following

\begin{lem}\label{lem1}Assume that $\mu\in M^1(G)$, $s\in I_{\mu}$ and $\Gamma=[\supp\mu]$ satisfies condition $(i-p)$.
Then there exists $c>0$ such that $\Q_x^{s, *}\le c \Q^{s, *}$ for every $x\in\Sp$. Moreover the
constant $c$ does not depend on $x\in\Sp$.
\end{lem}
\begin{proof}
We can repeat the argument from Section 3 in \cite{gaBDG}.
\end{proof}

\begin{lem}\label{Clem}
Assume that $\mu\in M^1(G)$, $s\in I_{\mu}$ and $\Gamma=[\supp\mu]$ satisfies condition $(i-p)$.
Then for every $x\in\Sp$ we have
\begin{align}
  \label{Clem1}  &\Q_x^{s,*}\l(\l\{\o\in\O: \exists\ C>0\ \forall\ n\in\N\ \ \ |S_n(\o)x|\ge C\|S_n(\o)\|\r\}\r)=1, \
    \ \ \mbox{and}\\
\label{Clem2}&\Q^{s, *}\l(\l\{\o\in\O: \exists\ C>0\ \forall\ n\in\N\ \ \ |S_n(\o)x|\ge
C\|S_n(\o)\|\r\}\r)=1.
\end{align}
\end{lem}
\begin{proof} Observe that \eqref{Clem2} implies \eqref{Clem1}. Indeed, let
$$
Z_x =\{\o\in\O: \exists\ C>0\ \forall\ n\in\N\ \ \ |S_n(\o)x|\ge
C\|S_n(\o)\|\},$$ and let $Z_x^c$ be the complement of $Z_x$. Then
by Lemma \ref{lem1}
$$
\Q^{s,*}_x(Z_x^c)\leq c\Q^{s,*}(Z_x^c)=0.$$
 The proof of \eqref{Clem2} is adapted from \cite{K}.
Conditions \eqref{con} yields the existence of $n_0\in\N$ and
$0<\tau<1$ such that
\begin{align}\label{defp}
    p=\P^*\l(\{\o\in\O: S_{n_0}(\o)(i, j)>\tau,\ \mbox{for all $1\le i, j\le d$}\}\r)>0,
\end{align}
where $\P^*=\mu_*^{\otimes\N}$. First of all we need to show that
\begin{align}
    \label{Tfinitx}&\Q^{s, *}_x(\{\o\in\O: T(\o)<\8\})=1, \ \ \ \mbox{for every $x\in\Sp$ and}\\
    \label{Tfinit}&\Q^{s, *}(\{\o\in\O: T(\o)<\8\})=1, \ \ \ \mbox{where}\ \ \ \ T(\o)=\min\{n\ge n_0: S_{n_0}(\te^{n-n_0}(\o))\in G^{\circ}\}.
\end{align}
Notice that \eqref{Tfinitx} immediately gives \eqref{Tfinit}, since the event $\{T<\8\}$ does not
depend on $x\in\Sp$ and $\Q^{s,*} = \int_{\Sp} \Q_x^{s, *}\pi^s_*(dx)$. 

Assume for a moment that \eqref{Tfinit} holds and prove
\eqref{Clem2}. If $x\in\Sp$ such that $x>0$ then for any $a\in G$
we have
$$|ax|\ge d^{-1/2}\sumi 1 d (ax)_i\ge d^{-1/2}\min_{1\le i\le d}x_i\sumij 1 d a(i, j)\ge d^{-1/2}\min_{1\le i\le d}x_i\|a\|,$$
hence \eqref{Clem2} holds with $C=d^{-1/2}\min_{1\le i\le
d}x_i>0$. Now fix an arbitrary $x\in\Sp$ and let
$\O_1=\{T<\8\}\subseteq\O$. By assumption, $\Q^{s, *}(\O_1)=1$.
It is easy to see
that $S_T(\o^*)x>0$ for $\o^*\in\O_1$. Thus for any $n\ge T$ and $\o^*\in\O_1$ we have
\begin{align*}
    |S_n(\o^*)x|=|S_{n-T}(\te^{T}(\o^*))S_T(\o^*)x|&\ge d^{-1/2}\min_{1\le i\le
    d}(S_T(\o^*)x)_i\|S_{n-T}(\te^{T}(\o^*))\|\\
    &\ge d^{-1/2}\frac{\min_{1\le i\le
    d}(S_T(\o^*)x)_i}{\|S_T(\o^*)\|}\|S_n(\o^*)\|.
\end{align*}
It implies that  $|S_n(\o^*)x|\ge C_{T, x}(\o^*)\|S_n(\o^*)\|$
holds with the constant $C_{T, x}(\o^*)>0$ independent of $n\ge
T$, for every $\o^*\in\O_1$.  Recall that $G$ is the
multiplicative semigroup of $d\times d$ invertible matrices with
positive entries such that every row and every column contains a
strictly positive element. Now take $n\le T$ and notice that
$C_{n, x}(\o^*)=\frac{|S_n(\o^*)x|}{\|S_n(\o^*)\|}>0$, for every
$\o^*\in\O_1$ by the definition of $G$ and $x\in\Sp$. Therefore,
we take $C(\o^*)=\min\{C_{1, x}(\o^*),\ldots,C_{T, x}(\o^*)\}>0$,
and \eqref{Clem2} follows.

We need only to prove \eqref{Tfinitx}. In this purpose we define
the events
\begin{align*}
    E_k=\{\o\in\O: S_{n_0}(\te^{k}(\o))(i, j)\ge\tau,\ \mbox{for all $1\le i, j\le
    d$}\}, \ \ k\in\N.
\end{align*}
We show that there exists $\ga\in[0, 1)$ such that for all $l\in\N$
\begin{align}\label{qtfin}
    \Q^{s,*}_x(\{T > ln_0\})\le\Q^{s,*}_x(\{E_{jn_0}\ \mbox{does not occur for any}\ 0\le j<l\})\le \ga^l.
\end{align}
Then \eqref{qtfin} with Borel--Cantelli lemma yield $\Q^{s,*}_x(\{T<\8\})=1$. In fact it is enough
to show that
\begin{align}
  \label{ess} \Q^{s,*}_x(E_{0}^c\cap\ldots\cap E_{(l-1)n_0}^c)&\le\ga\Q^{s,*}_x(E_{0}^c\cap\ldots\cap E_{(l-2)n_0}^c)\le\ga^2\Q^{s,*}_x(E_{0}^c\cap\ldots\cap
   E_{(l-3)n_0}^c)\\
  \nonumber &\le\ldots\mbox{ and inductively }\ldots\le\ga^l.
\end{align}

Let $r_{s}=\frac{\inf_{x\in\Sp}e^s(x)}{\sup_{x\in\Sp}e^s(x)}$.
Then
\begin{multline}\label{loes}
\Q^{s,*}_x(E_{0}^c\cap\ldots\cap E_{(l-2)n_0}^c\cap E_{(l-1)n_0})\\
=\int_{\O}\I{E_{0}^c\cap\ldots\cap
E_{(l-2)n_0}^c\cap E_{(l-1)n_0}}(\o^*)q_{ln_0}^s(x,S_{ln_0}(\o^*))\mu^{*ln_0}(d\o)\\
\ge\frac{r_s\tau^s}{d^{s/2}\ka^{n_0}(s)}\int_{\O}\I{E_{0}^c\cap\ldots\cap
E_{(l-2)n_0}^c}(\o^*)\I{E_{(l-1)n_0}}(\o^*)q_{(l-1)n_0}^s(x,S_{(l-1)n_0}(\o^*))\mu^{*ln_0}(d\o)\\
=\frac{r_s\tau^s}{d^{s/2}\ka^{n_0}(s)}\P^*(E_{(l-1)n_0})\Q^{s,*}_x(E_{0}^c\cap\ldots\cap
E_{(l-2)n_0}^c)=\frac{pr_s\tau^s}{d^{s/2}\ka^{n_0}(s)}\Q^{s,*}_x(E_{0}^c\cap\ldots\cap
E_{(l-2)n_0}^c),
\end{multline}
since by \eqref{coc} we have the following lower bound
\begin{multline*}
\I{E_{(l-1)n_0}}(\o^*)q_{ln_0}^s(x,S_{ln_0}(\o^*))\\
=\I{E_{(l-1)n_0}}(\o^*)q_{(l-1)n_0}^s(x,S_{(l-1)n_0}(\o^*))q_{n_0}^s(S_{(l-1)n_0}(\o^*)\ra
x,S_{n_0}(\te^{(l-1)n_0}(\o^*)))\\
\ge \frac{r_s}{\ka^{n_0}(s)}\I{E_{(l-1)n_0}}(\o^*)q_{(l-1)n_0}^s(x,S_{(l-1)n_0}(\o^*))|S_{n_0}(\te^{(l-1)n_0}(\o^*))(S_{(l-1)n_0}(\o^*)\ra x)|^s\\
\ge
\frac{r_s}{d^{s/2}\ka^{n_0}(s)}\I{E_{(l-1)n_0}}(\o^*)q_{(l-1)n_0}^s(x,S_{(l-1)n_0}(\o^*))\l(\sumi 1
d
S_{n_0}(\te^{(l-1)n_0}(\o^*))(S_{(l-1)n_0}(\o^*)\ra x)_i\r)^s\\
\ge
\frac{r_s\tau^s}{d^{s/2}\ka^{n_0}(s)}\I{E_{(l-1)n_0}}(\o^*)q_{(l-1)n_0}^s(x,S_{(l-1)n_0}(\o^*)).
\end{multline*}
Let $0<\ga_s=\min\l\{1, \frac{pr_s\tau^s}{d^{s/2}\ka^{n_0}(s)}\r\}$. For $\ga=1-\ga_s\in[0, 1)$, by
\eqref{loes}, we obtain that
\begin{align*}
\Q^{s,*}_x(E_{0}^c\cap\ldots\cap E_{(l-2)n_0}^c\cap
E_{(l-1)n_0}^c)&\le\ga\Q^{s,*}_x(E_{0}^c\cap\ldots\cap E_{(l-2)n_0}^c\cap
  G_{(l-1)n_0})\\
&=\ga\Q^{s,*}_x(E_{0}^c\cap\ldots\cap E_{(l-2)n_0}^c).
\end{align*}
 This finishes the proof of \eqref{ess} and completes
the proof of the lemma.
\end{proof}

\begin{lem}\label{birkhofflem}
Assume that $\mu\in M^1(G)$, $s\in I_{\mu}$ and $\Gamma=[\supp\mu]$ satisfies condition $(i-p)$.
Assume additionally that $\int_G\|a\|^{s}\log^+\|a\|\mu(da)<\8$. Then for any $x\in\Sp$
\begin{align}\label{birkhoff}
    \limn \frac{1}{n}\log|S_n(\o)x|=\limn \frac{1}{n}\log\|S_n(\o)\|=\al(s),\ \ \mbox{ $\Q^{s,
    *}_x$ and $\Q^{s, *}$ a.s.,}
\end{align}
where
\begin{align}\label{constal}
    \al(s)=\int_{\Sp}\int_G\log|ax|q_1^s(x, a)\mu_*(da)\pi^s_*(dx).
\end{align}

\end{lem}
\begin{proof}
We show that $f(x, \o)=\log|S_1(\o)x|$ is $^aQ^{s, *}$ integrable. Observe that there exists
$0<\de<1$ such that
$$0<|ax|<\de\ \Longrightarrow\ |ax|^{s}\log|ax|^{-1}\le1.$$
Then
\begin{align*}
    ^aQ^{s, *}(|f|)&=\int_{\Sp}\int_{\O}|\log|S_1(\o)y||\de_x(dy)\Q_x^{s,*}(d\o)\pi^{s}_*(dx)\\
    &=\int_{\Sp}\int_{G}|ax|^{s}|\log|ax||\frac{e^{s}(a\ra
    x)}{\ka(s)e^{s}(x)}\mu_*(da)\pi^{s}_*(dx)\\
    &\le C_{s}\int_{\Sp}\int_{G}|ax|^{s}|\log|ax||\mu_*(da)\pi^{s}_*(dx)\\
\le C_{s}\int_G\|a\|^{s}\log^+\|a\|\mu(da)&+C_{s}\int_{\Sp}\int_{G}|ax|^{s}\log^-|ax|\mu_*(da)\pi^{s}_*(dx)\\
\le C_{s}\int_G\|a\|^{s}\log^+\|a\|\mu(da)&+C_{s}\mu_*\otimes\pi^{s}_*(\{(a, x)\in G\times\Sp: 0<|ax|<\de\})\\
&+C_{s}\log\l(\frac{1}{\de}\r)\mu_*\otimes\pi^{s}_*(\{(a, x)\in G\times\Sp: \de<|ax|\le1\})\\
\le C_{s}\int_G\|a\|^{s}\log^+\|a\|\mu(da)&+C_{s}\l(1+\log\l(\frac{1}{\de}\r)\r)<\8.
\end{align*}
Hence in view of Remark \ref{remerg}, on the one hand, by the
Birkhoff ergodic theorem (applied to $^aQ^{s, *}$ and $^a\theta$)
we obtain
\begin{align*}
^aQ^{s, *}\l(\l\{(x, \o)\in{^a\O}: \lim_{n\to\8} \frac 1n \log|S_n(\o)x|  = \lim_{n\to\8} \frac 1n
\cdot \sum_{k=0}^{n-1} f\circ \gt^k(x,\o) = {^aQ^{s, *}}(f) = \al(s)\r\}\r)=1.
\end{align*}
On the other hand by the Kingman subadditive ergodic theorem
(applied to $Q^{s, *}$ and $\theta$) we have
\begin{align*}
\Q^{s, *}\l(\l\{\o\in\O: \lim_{n\to\8} \frac 1n \log\|S_n(\o)\|=\al_{s}\r\}\r)=1.
\end{align*}
Define $\O'=\l\{\o\in\O: \exists\ C>0\ \forall\ n\in\N\ |S_n(\o)x|\ge C\|S_n(\o)\|\mbox{ and
}\lim_{n\to\8}\frac1n\log\|S_n(\o)\|=\al_{s}\r\}$, for every $x\in\Sp$. By Lemma \ref{Clem} and
calculations stated above we know that $\Q^{s, *}(\O')=1$. Fix arbitrary $x\in\Sp$, take any
$\o^*\in\O'$ and notice that
\begin{align*}
    0<C_x(\o^*)\le\frac{|S_n(\o^*)x|}{\|S_n(\o^*)\|}\le1,
\end{align*}
imply
\begin{align*}
    \frac 1n \log C_x(\o^*)+\frac 1n
    \log\|S_n(\o^*)\|\le\frac 1n\log\frac{|S_n(\o^*)x|}{\|S_n(\o^*)\|}+\frac 1n\log\|S_n(\o^*)\|\le\frac
    1n\log\|S_n(\o^*)\|.
\end{align*}
Since  $\limn\frac{1}{n}\log C_x(\o^*)=0$ we have
\begin{align*}
    \Q^{s, *}\l(\l\{\o\in\O: \lim_{n\to\8} \frac 1n \log|S_n(\o)x|=\al_{s}\r\}\r)=1.
\end{align*}
And so, in view of Lemma \ref{lem1},
\begin{align*}
    \Q^{s, *}_x\l(\l\{\o\in\O: \lim_{n\to\8} \frac 1n \log|S_n(\o)x|=\al_s\r\}\r)=1.
\end{align*}
for all $x\in\Sp$ (by considering complements). Since $^aQ^{s,*} =
\int_{\Sp} \delta_x \otimes
    \Q_x^{s, *}\pi^s_*(dx)$ we get $\al(s)=\al_s$ and Lemma \ref{birkhofflem} follows.
\end{proof}

\subsection{Kesten's renewal theorem}\label{KRT2}
For $x\in \Sp$ and $\o\in\O$ define $X_0(\omega)=x$, and for $n\in\N$
$$X_n(\o)=g_n(\o)\cdot X_{n-1}(\o)=S_n(\o)\cdot x,$$
and
\begin{align*}
    V_n(\o) = \log|S_n(\o)x| = \sum_{i=1}^n U_i(\o),\ \ \mbox{ where }\ \ U_i(\o) = \log|g_i(\o)
    X_{i-1}(\o)|.
\end{align*}
Let $F(dt|x,y)$ be the conditional law of $U_1$, given $X_0=x$,
$X_1=y$, i.e.
$$
\Q_x^{s, *}\l( X_1\in A, U_1\in B \r) = \int_A\int_B F(dt|x,y)Q^s_*(x,dy).
$$

A  function $g:\Sp\times \R \to \R$ is called direct Riemann
integrable ($d\mathcal{R}i$),  if it is $\B or(\Sp)\times\B
or(\R)$ measurable and for every fixed $x\in \Sp$ and $0<L<\8$ the
function $t\mapsto g(x,t)$ is Riemann integrable on $[-L,L]$, and
satisfies
\begin{equation}
\label{dri1} \sum_{k=0}^\8 \sum_{l=-\8}^\8 (k+1) \sup \l\{  |g(x,t)|:\; x\in C_{k+1}\setminus
C_k,\mbox{ and } t\in[l, l+1] \r\}<\8,
\end{equation}
where
\begin{equation}
\label{dri2}C_k=\l\{ x\in \Sp:\; \Q_x^{s, *}\l(\l\{\frac{V_m}{m} \ge \frac 1k,\ \ \mbox{for all
$m\ge k$}\r\} \r) \ge \frac 12 \r\},\ \ \ \mbox{for all $k\in\N$}.
\end{equation}
For the reader's convenience we formulate Kesten's renewal theorem
\cite{K1}.
\begin{thm} \label{kesten} Assume the following conditions are satisfied:
\begin{itemize}
\item {\textbf{Condition I.1}}
There exists $\pi^s_*\in M^1(\Sp)$ such that $\pi^s_* Q^s_*=\pi^s_*$ and for every open set
$U\subseteq\Sp$ with $\pi^s_*(U)>0$, $\Q_x^{s,*}(X_n\in U \mbox{ for some } n\in\N)=1$ for every
$x\in\Sp$.
\item {\textbf{Condition I.2}}
$$
\int_{\Sp}\int_{\Sp}\int_{\R}|t|F(dt|x,y) Q^s_*(x,dy)\pi^s_*(dx)<\8,
$$ and for all $x\in\Sp$,
\begin{equation}\label{constal1}
\lim_{n\to\8} \frac{V_n}{n} = \al(s) = \int t F(dt|x,y)Q^s_*(x,dy)\pi^s_*(dx)>0\ \ \ \mbox{$
\Q_x^{s,*}$ -- a.e.}.
\end{equation}
\item {\textbf{Condition I.3}}
 There exists a sequence $\{\z_i\}\subset\R$ such that the group
generated by $\z_i$ is dense in $\R$ and such that for each $\z_i$ and $\la>0$ there exists
$y=y(\z_i,\la)\in \Sp$ with the following property: for each $\eps>0$, there exists an $A\in
\mathcal{B}or(\Sp)$ with $\pi^s_*(A)>0$ and $m_1,m_2\in \N$, $\tau\in\R$ such that for any $x\in A$
\begin{eqnarray}
\label{d1} &\Q_x^{s,*}\l\{ |X_{m_1}-y|<\eps, |V_{m_1} - \tau|\le \la \r\} >0,\\
\label{d2} &\Q_x^{s,*}\l\{ |X_{m_2}-y|<\eps, |V_{m_2} - \tau-\z_i|\le \la \r\} >0.
\end{eqnarray}
\item {\textbf{Condition I.4}}
For each  fixed $x\in\Sp$, $\eps>0$ there exists $r_0=r_0(x,\eps)>0$ such that for all real valued
functions $f$ measurable with respect to $\mathcal{B}or\l((\Sp\times \R)^{\N}\r)$ and for all
$y\in\Sp$ with $|x-y|<r_0$ one has:
\begin{eqnarray*}\E_x^{s,*} f(X_0,V_0,X_1,V_1,\ldots)&\le& \E_y^{s,*} f^{\eps}(X_0,V_0,X_1,V_1,\ldots) +\eps|f|_\8,\\
\E_y^{s,*} f(X_0,V_0,X_1,V_1,\ldots)&\le&
\E_x^{s,*} f^{\eps}(X_0,V_0,X_1,V_1,\ldots) +\eps|f|_\8,
\end{eqnarray*}
where $f^{\eps}(x_0,v_0,x_1,v_1,\ldots) =\sup\l\{ f(y_0,u_0,y_1,u_1,\ldots):\ \forall\ i\in \N\
|x_i-y_i|+|v_i-u_i|<\eps \r\}.$
\end{itemize}
If a function $g:\Sp\times\R\mapsto\R$ is jointly continuous and $(d\mathcal{R}i)$, then for every
$x\in \Sp$
$$
\lim_{t\to\8} \E_x^{s,*} \l(\sum_{n=0}^\8 g(X_n,t-V_n)\r) = \frac 1{\al(s)}\int_{\Sp}\l(\int_\R
g(y,x)dx\r)\pi^s_*(dy),
$$ for $\al(s)$ defined in \eqref{constal1}.
\end{thm}

In the next four subsections we indicate how the material
developed in Section \ref{TO} and \ref{KRT1}, under the hypotheses of Theorem
\ref{thm2}, may be used to check the assumptions of Theorem
\ref{kesten}. From now we will work with the measures
$\Q_x^{\c,*}$ for $x\in\Sp$, where $\c>0$ solves equation
$\ka(\c)=\frac{1}{N}$. Such $\c>0$ exists since $\ka(s)$ is
strictly $\log$--convex and $\lim_{s\to s_{\8}}\ka(s)>\frac 1N$,
(see Theorem \ref{thm2} and Theorem \ref{thmGL}). We are going to prove
that Conditions I.1--I.4 are satisfied for $s=\c$.

\subsection{Condition I.1.}
\begin{proof}[Proof of Condition I.1.]
Theorem \ref{thmGL} with Breiman's strong law of large numbers \cite{gaBre} allow us to repeat the argument contained in Section 5 in \cite{gaBDG}.
\end{proof}
\subsection{Condition I.2.} 
\begin{proof}[Proof of Condition I.2.]
We know that $\int_G\|a\|^{\c}\log^+\|a\|\mu(da)<\8$, hence
$$
\int_{\Sp}\int_{\Sp}\int_{\R}|t|F(dt|x,y)
Q^{\c}_*(x,dy)\pi^{\c}_*(dx)=\int_{\Sp}\int_{\O}|\log|ax||q_1^{\c}(x,
a)\mu_*(da)\pi^{\c}_*(dx)<\8,
$$
by the arguments of Lemma \ref{birkhofflem} applied to $s=\c$. The
only point remaining concerns the positivity of $\al(\c)$ defined
in Lemma \ref{birkhofflem} (see also \eqref{constal1}).

Notice that if $\eps>0$ is sufficiently small, then for every $t\in(\c-\eps, \c)$, we have
$\ka(t)<\ka(\c)$, since $\ka(s)$ is strictly $\log$--convex and $\lim_{s\to s_{\8}}\ka(s)>\frac
1N$, (see Theorem \ref{thm2} and Theorem \ref{thmGL}). Fix $t\in(\c-\eps, \c)$ such that $\c/t\le4/3$
and take $\ga>0$ such that $\ka(t)e^{\ga}<\ka(\c)$. In view of inequality \eqref{l2a}, there is
$C>0$ such that
\begin{align*}
    \int_{G}\|a\|^t\mu^{*n}_*(da)\le C\ka^n(t)e^{\ga n/3},\ \ \mbox{ for every $n\in\N$},
\end{align*}
since $1\le e^{\ga/3}$. Fix $x\in\Sp$. Then for $\de=\ga/3$ we
have
\begin{align*}
    \mu^{\ast n}_*(\{a\in G: |ax|^t>e^{-\de n}\})\le e^{\de n}\int_{G}|ax|^t\mu^{\ast n}_*(da)\le
    C\ka^n(t)e^{2\ga n/3},
\end{align*}
Now let  $\rho=\ga/6$. Then
\begin{multline*}
    \Q_x^{\c,*}(\{\o\in\O: |S_n(\o)x|^t<e^{\rho n}\})=\int_G\I{\{a\in G: |ax|^t<e^{\rho n}\}}q_n^{\c}(a, x)\mu^{\ast
    n}_*(da)\\
    \le\frac{C}{\ka^n(\c)}\int_G\I{\{a\in G: |ax|^t<e^{\rho n}\}}|ax|^{\c}\mu^{\ast
    n}_*(da)\\
    \le\frac{C}{\ka^n(\c)}\int_G\I{\{a\in G: |ax|^t<e^{-\de n}\}}|ax|^{\c}\mu^{\ast
    n}_*(da)+\frac{C}{\ka^n(\c)}\int_G\I{\{a\in G: e^{-\de n}\le|ax|^t<e^{\rho n}\}}|ax|^{\c}\mu^{\ast
    n}_*(da)\\
    \le\frac{C\ka^n(t)}{\ka^n(\c)}\ra\frac{1}{\ka^n(t)}e^{-\frac{\c-t}{t}\de n}\int_G|ax|^{t}\frac{e^{t}(a\ra x)}{e^t(x)}\mu^{\ast
    n}_*(da)+\frac{C}{\ka^n(\c)}\mu^{\ast n}_*(\{a\in G: |ax|^t>e^{-\de n}\})e^{\frac{\rho n \c}{t}}\\
    \le Ce^{-\l(\ga+\frac{\c-t}{t}\de\r)n}+Ce^{-\ga n}e^{2\ga n/3}e^{\rho n\c/t}\\
\le Ce^{-\l(\ga+\frac{\c-t}{t}\de\r)n}+Ce^{-\ga n/3+2\ga n/9}=Ce^{-\be n},
\end{multline*}
 for some $\be>0$. Thus
$$\sum_{n\in\N} \Q_x^{\c,*}\l(\l\{\o\in\O: \log|S_n(\o)x|<\frac{\rho n}{t}\r\}\r)<\8.$$
Therefore, by the Borel--Cantelli lemma we obtain that for every $x\in\Sp$
$$\Q_x^{\c,*}\l(\l\{\o\in\O: \liminf_{n\to\8}\frac{\log|S_n(\o)x|}{n}\ge\frac{\rho}{t}>0\r\}\r)=1.$$
This shows that $\al(\c)>0$ $\Q_x^{\c,*}$ a.s. for every $x\in\Sp$ and finishes the proof of
Condition I.2..
\end{proof}
\subsection{Condition I.3.}
\begin{proof}[Proof of Condition I.3.]
Proposition \ref{gr} and Theorem \ref{thmGL} allow us to use arguments from Section 5 in \cite{gaBDG}.
\end{proof}
\subsection{Condition I.4.}
\begin{proof}[Proof of Condition I.4.]
The proof is a consequence of Lemma \ref{Clem} and the argument given by Kesten \cite{K}.
\end{proof}

\subsection{Direct Riemann integrability}\label{KRT7}
Now we derive an interesting criterium which significantly simplifies condition \eqref{dri1}.
\begin{lem}
\label{gadrilem} Assume that the hypotheses of Theorem \ref{thm2} are satisfied. If $h$ is any
bounded and continuous function on $\Sp\times\R$ which satisfies
\begin{equation}
\label{dri}
 \sum_{l=-\8}^\8  \sup \l\{  |h(x,t)|:\ x\in \Sp,\mbox{ and } t\in [l,l+1] \r\}<\8,
\end{equation}
then $h$ is direct Riemann integrable i.e. it satisfies condition
\eqref{dri1}.
\end{lem}
\begin{proof} We give only a sketch of the proof, for more details we refer to \cite{gaBDG}.
First of all we prove that $C_k=\Sp$, for some sufficiently large
$k\in\N$, ($C_k$ was defined in \eqref{dri2}). Then obviously
\eqref{dri} implies \eqref{dri1}. There is a finite number $N_1$
of points such that $\Sp\subseteq\bigcup_{i=1}^{N_1}B(x_i, 2)$,
since $\Sp$ is compact. Let
\begin{align*}
  \O'=\l\{\limn\frac{\log|S_nx_i|}{n}=\al(\c)>0,\mbox{ and }\exists C>0\ \forall n\in\N\ \
|S_nx_i|\ge C\|S_n\|,\mbox{ for all $1\le i\le N_1$}\r\}.
\end{align*}
Then $\Q^{\c, *}(\O')=1$, by Lemma \ref{Clem} and
\ref{birkhofflem}. Take any $y\in\Sp$, then there exists $1\le
i\le N_1$ such that $y\in B(x_i, 2)$. This implies  the existence of $m_0\in\N$ such that
\begin{align*}
    \Q^{\c,*}\l(\l\{\o\in\O: \frac{\log|S_n(\o)y|}{n}>\al(\c)/2,\mbox{ for all $n\ge m_0$ }\r\}\r)\ge1-\frac{1}{2c},
\end{align*}
with the constant $c>0$ defined in Lemma \ref{lem1}. Taking any $1/k\le\min\{\al(s)/2, 1/m_0\}$ Lemma \ref{gadrilem} follows.
\end{proof}

\section{Proof of the main Theorem}\label{GMT} In this section we give a detailed  proof of Theorem \ref{thm2}. For that
we consider the following smooth version of $\P(\{\iss R u>t\})$
\begin{align}\label{funG}
    G(u, t)=\frac{1}{e^te^{\c}(u)}\int_{0}^{e^t}r^{\c}\P(\{\iss R u>r\})dr, \ \ \ \mbox{where $(u, t)\in\Sp\times\R$,}
\end{align}
where $R\in\R^d_+$ solves equation \eqref{eq}. Let
$\mathbf{B}(\Sp\times\R)$ be the space of all bounded measurable
functions on $\Sp\times\R$. Define a linear operator $\Te:
\mathbf{B}(\Sp\times\R)\mapsto\mathbf{B}(\Sp\times\R)$ given by
the formula
\begin{align*}
    \Te f(u, t)&=\E^{\c,*}_u\l(f(X_1, t-V_1)\r)\\
    &=\frac{1}{\ka(\c)}\int_{\O}f(S_1(\o^*)\ra
u, t-\log|S_1(\o^*)u|)\frac{e^{\c}(S_1(\o^*)\ra u)}{e^{\c}(u)}|S_1(\o^*)u|^{\c}\P(d\o).
\end{align*}
 Observe that
for every $n\in\N$
\begin{align*}
    \Te^nf(u, t)=\E^{\c,*}_u\l(f(X_n, t-V_n)\r).
\end{align*}
First we express $G(u, t)$ as a potential of a function $g(u, t)$ that turns
out later on to be direct Riemann integrable.
\begin{lem}\label{funG0lem}
Assume that the hypotheses of Theorem \ref{thm2} are satisfied. Let $G(u, t)$ be the function defined in
\eqref{funG}, then
\begin{align}
   \label{funG0} G_0(u, t)=\frac{N}{e^te^{\c}(u)}\int_{0}^{e^t}r^{\c}\P(\{\iss {AR}
    u>r\})dr=\Te G(u, t),\ \ \ \mbox{and}
\end{align}
\begin{align}
    \label{limG}\limn\Te^nG(u, t)=\limn\E^{\c,*}_u\l(G(X_n, t-V_n)\r)=0.
\end{align}
Moreover,
\begin{align}
   \label{limG1} G(u, t)=\sum_{n=0}^{\8}\Te^{n}g(u, t),\ \ \ \mbox{where}
\end{align}
\begin{align}
    \label{fung}  g(u, t)=\frac{1}{e^te^{\c}(u)}\int_{0}^{e^t}r^{\c}\l(\P(\{\iss R u>r\})-N\P(\{\iss {AR}
    u>r\})\r)dr.
\end{align}

\end{lem}
\begin{proof} First of all we show $G_0(u, t)=\Te G(u, t)$. Indeed,
\begin{align*}
   G_0(u, t)&=\frac{N}{e^te^{\c}(u)}\int_{0}^{e^t}r^{\c}\P(\{\iss {R}
    {A^*\ra u}|A^*u|>r\})dr\\
    &=\E\l(\frac{N}{e^te^{\c}(u)}\int_{0}^{e^t}r^{\c}\I{\l(\frac{r}{|A^*u|}, \8\r)} (\iss {R}
    {A^*\ra u})dr\r)\\
    &=\E\l(\frac{N}{\frac{e^t}{|A^*u|}e^{\c}(u)}\int_{0}^{\frac{e^t}{|A^*u|}}r^{\c}\I{\l(r, \8\r)} (\iss {R}
    {A^*\ra u})|A^*u|^{\c}dr\r)\\
    &=\E\l(\frac{1}{\frac{e^t}{|A^*u|}e^{\c}(A^*\ra u)}\int_{0}^{\frac{e^t}{|A^*u|}}r^{\c}\I{\l(r, \8\r)} (\iss {R}
    {A^*\ra u})dr\frac{1}{\ka(\c)}\frac{e^{\c}(A^*\ra u)}{e^{\c}(u)}|A^*u|^{\c}\r)\\
    &=\Te G(u, t).
\end{align*}
Now we have
\begin{align*}
    \Te^nG(u, t)&=\E^{\c,*}_u\l(G(X_n, t-V_n)\r)=\E^{*}\l(G(S_n\ra u, t-\log|S_nu|)\frac{1}{\ka^n(\c)}\frac{e^{\c}(S_n\ra
    u)}{e^{\c}(u)}|S_nu|^{\c}\r)\\
&=N^n\E^{*}\l(\frac{|S_nu|}{e^te^{\c}(S_n\ra
    u)}\int_0^{\frac{e^t}{|S_nu|}}r^{\c}\I{(r, \8)}(\iss{R}{S_n\ra u})\frac{e^{\c}(S_n\ra
    u)}{e^{\c}(u)}|S_nu|^{\c}dr\r)\\
   & =N^n\E^{*}\l(\frac{|S_nu|^{\c+1}}{e^te^{\c}(u)}\int_0^{\frac{e^t}{|S_nu|}}r^{\c}\I{(r, \8)}(\iss{R}{S_n\ra u})dr\r)\\
    &=N^n\E^{*}\l(\frac{|S_nu|^{\c+1}}{e^te^{\c}(u)}\int_0^{\frac{e^t}{|S_nu|}}r^{\c}\I{(|S_nu|r,
    \8)}(\iss{S_n^*R}{u})dr\r)\\
    &=\frac{N^n}{e^te^{\c}(u)}\int_0^{e^t}r^{\c}\E^{*}\l(\I{(r,
    \8)}(\iss{S_n^*R}{u})dr\r)\\
    &=\frac{N^n}{e^te^{\c}(u)}\int_0^{e^t}r^{\c}\E\l(\I{(r,
    \8)}(\iss{A_1\raa A_nR}{u})dr\r),
\end{align*}
where $S_n=A_n\raa A_1$. By the continuity of $I_{\mu}\ni s\mapsto\ka(s)$ (see Theorem \ref{thmGL}) we
can find  $p<\c$, such that $\ka(p)=\frac{1-\eps}{N}$, for some $\eps>0$, then
\begin{align*}
    \E\l(\I{(r,
    \8)}(\iss{A_1\raa A_nR}{u})\r)\le\frac{\E\l(\|A_1\raa A_n\|^p\r)\E\l(|R|^p\r)}{r^p}\le\frac{C\ka^n(p)\E\l(|R|^p\r)}{r^p}.
\end{align*}
This implies that
\begin{align*}
    \Te^nG(u, t)&=\frac{N^n}{e^te^{\c}(u)}\int_0^{e^t}r^{\c}\E\l(\I{(r,
    \8)}(\iss{A_1\raa A_nR}{u})dr\r)\\
    &\le\frac{CN^n}{e^te^{\c}(u)}\int_0^{e^t}r^{\c-p}\ka^n(p)\E\l(|R|^p\r)dr\\
    &\le\frac{CN^n}{e^te^{\c}(u)}\E\l(|R|^p\r)\l(\frac{1-\eps}{N}\r)^n\int_0^{e^t}r^{\c-p}dr\\
    &\le\frac{C\E\l(|R|^p\r)}{e^{\c}(u)}e^{t(\c-p)}(1-\eps)^n\ _{\overrightarrow{n\to\8}}\ 0.
\end{align*}
Now it is easy to see that for any $n\in\N$ we have
\begin{align*}
    G(u, t)=g(u, t)+\Te g(u, t)+\Te^2g(u, t)+\ldots+\Te^{n-1}g(u, t)+\Te^{n}G(u,
    t),
\end{align*}
and \eqref{limG1} follows. This completes the proof of Lemma \ref{limG}.
\end{proof}
Lemmas \ref{fung1lem} and \ref{fung2lem} below imply that $g(u, t)$ is
direct Riemann integrable. Lemmas \ref{contP}, \ref{nier1} and
\ref{nier2} contain some necessary technicalities.
\begin{lem}\label{contP}
Assume that the hypotheses of Theorem \ref{thm2} are satisfied. Then $\P(\{\iss Ru=r\})=0$, for
every $(u, r)\in\Sp\times\R^+\cup\{0\}$. Moreover, for every $r\ge 0$ the functions
\begin{align*}
    \Sd\ni u\mapsto \P(\{\iss Ru>r\}), \ \ \ \mbox{and}\ \ \ \ \Sd\ni u\mapsto \P(\{\iss {AR}{u}>r\}),
\end{align*}
are continuous.
\end{lem}
\begin{proof}
At the beginning, we assume that the law $\eta$ of $B$ is nonsingular, i.e. $\|\eta_s\|<1$.
Let $\nu$ be the law of $R$ and $\mu$ be the law of $A\in G$. Let $*$ be the classical convolution on $\Rd$. Moreover, we define $\xi=\mu*_G\nu$, where $\mu*_G\nu(D)=\int_G\int_{\Rd}\mathbf{1}_{D}(ax)\nu(dx)\mu(da)$ and $D\in\mathcal{B}or(\Rd)$. Obviously $\xi$ defines a probability measure on $\Rd$ which coincide with the distribution of $AR$. Notice that $\nu=\xi^{*N}*\eta$, since $R\stackrel{\mathcal{D}}{=}\sumi 1 N A_iR_i+B$, and observe that by the Lebesgue decomposition we obtain
\begin{multline*}
    \nu_a+\nu_s=\nu=(\xi_a+\xi_s)^{*N}*(\eta_a+\eta_s)=\\
    \sum_{n=0}^N\binom{N}{n}\xi_a^{*n}*\xi_s^{*(N-n)}*\eta_a
    +\sum_{n=1}^{N}\binom{N}{n}\xi_a^{*n}*\xi_s^{*(N-n)}*\eta_s\\
    +(\xi_s^{*N}*\eta_s)_a+(\xi_s^{*N}*\eta_s)_s,
\end{multline*}
and by its uniqueness $\nu_s=(\xi_s^{*N}*\eta_s)_s$. This gives $\|\nu_s\|\le\|\xi_s\|^N\|\eta_s\|$. Again by the Lebesgue decomposition and its uniqueness we have $\xi=\mu*_G\nu=\mu*_G\nu_a+\mu*_G\nu_s$, hence $\|\xi_s\|=\|(\mu*_G\nu)_s\|\le\|\mu*_G\nu_s\|\le\|\nu_s\|$. Now combining $\|\nu_s\|\le\|\xi_s\|^N\|\eta_s\|$ and $\|\xi_s\|\le\|\nu_s\|$ we get $\|\nu_s\|\le\|\nu_s\|^N\|\eta_s\|$, if $\|\nu_s\|>0$, then $1\le\|\nu_s\|^{N-1}\|\eta_s\|\le\|\eta_s\|<1$. This contradiction shows that $\|\nu_s\|=0$ hence $\nu$ is absolutely continuous with respect to the Lebesgue measure, which in turn implies that $\P(\{\iss Ru=r\})=0$, for every $(u, r)\in\Sp\times\R^+\cup\{0\}$.

If the law $\eta$ of $B$ is singular, i.e. $\|\eta_s\|=1$, then for
fixed $(u, r)\in\Sp\times\R^+\cup\{0\}$, we have $\P(\{\iss Ru=r\})=0$, since $\P(\{\iss
Bu=r\})=0$.

Now we prove that $\Sd\ni u\mapsto \P(\{\iss Ru>r\})$ is continuous. Take any $(u_n)_{n\in\N}\subseteq\Sp$ such that $\limn u_n=u\in\Sp$ and consider
\begin{align*}
    |\P(\{\iss R{u_n}>r\})-\P(\{\iss Ru>r\})|&\le \P(\{\iss R{u_n}>r,\mbox{ and }\iss R{u}\le r\})\\
&+\P(\{\iss R{u_n}\le r,\mbox{ and }\iss R{u}> r\}),
\end{align*}
then
\begin{align*}
    \P(\{\iss R{u}\le r<\iss R{u_n}\})&=\P(\{0\le r-\iss R{u}< \iss R{u_n}-\iss R{u}\})\\
    &\le\P(\{0\le r-\iss R{u}\le |R||u_n-u|\}),\ \ \ \mbox{and}\\
    \P(\{\iss R{u_n}\le r<\iss R{u}\})&=\P(\{\iss R{u_n}-\iss R{u}\le r-\iss R{u}<0\})\\
    &\le\P(\{-|R||u_n-u|\le r-\iss R{u}<0\}).
\end{align*}
If $|u_n-u|<1/m$, then $$|\P(\{\iss R{u_n}>r\})-\P(\{\iss Ru>r\})|\le\P(\{|\iss
Ru-r|\le|R||u_n-u|\})\le\P(\{|\iss Ru-r|\le|R|/m\}).$$ We also know that $\lim_{m\to\8}\P(\{|\iss
Ru-r|\le|R|/m\})=\P(\{\iss Ru=r\})=0$, hence
$$\limn|\P(\{\iss R{u_n}>r\})-\P(\{\iss Ru>r\})|=0.$$
The same arguments work for $u\mapsto \P(\{\iss {AR}{u}>r\})$, since $A\in G$ is independent of
$R$.
\end{proof}
\begin{lem}\label{fung1lem}
Under the assumptions of Theorem \ref{thm2}, there exists $0<\be_1<1$ such that for every
$\be\in[0, \be_1)$, there is a finite constant $C_{\be}>0$, such that for every $(u,
t)\in\Sp\times\R$ we have
\begin{multline}\label{fung1}
    g_1(u, t)=\frac{1}{e^te^{\c}(u)}\int_{0}^{e^t}r^{\c}\l|\P\l(\l\{\max_{1\le i\le N}\iss {A_iR_i}
    u>r\r\}\r)-N\P(\{\iss {AR}
    u>r\})\r|dr\\
    \le C_{\be}e^{-\be|t|},
\end{multline}
and
\begin{multline}\label{fung1a}
    \int_{0}^{\8}\l(N\P(\{\iss {AR}
    u>r\})-\P\l(\l\{\max_{1\le i\le N}\iss {A_iR_i}
    u>r\r\}\r)\r)r^{\c+\be-1}dr\\
    =\frac{1}{\c+\be}\E\l(\sum_{i=1}^N\iss{A_iR_i}{u}^{\c+\be}-\l(\max_{1\le i\le N}\iss{A_iR_i}{u}\r)^{\c+\be}\r).
\end{multline}
Moreover,
$\Sp\times\R\ni(u, t)\mapsto g_1(u, t)$ is continuous.
\end{lem}
In the proof we extend the approach developed in \cite{gaJel1}.
\begin{proof}
Let $\be_1\in(0, \min\{1, \c/2\})$ and take any $0\le\be<\be_1$. Then for every $t>0$
\begin{align*}
    I_1&=e^{-\be t}e^{-(1-\be)t}\int_{0}^{e^t}r^{\c}\l|\P\l(\l\{\max_{1\le i\le N}\iss {A_iR_i}
    u>r\r\}\r)-N\P(\{\iss {AR}
    u>r\})\r|dr\\
&\le e^{-\be t}\int_{0}^{e^t}r^{\c+\be-1}\l|\P\l(\l\{\max_{1\le i\le N}\iss {A_iR_i}
    u>r\r\}\r)-N\P(\{\iss {AR}
    u>r\})\r|dr.
\end{align*}
Now observe that $N\P(\{\iss {AR}
    u>r\})\ge\P\l(\l\{\max_{1\le i\le N}\iss {A_iR_i}
    u>r\r\}\r)$, then
    \begin{align*}
       \int_{0}^{1}\l(N\P(\{\iss {AR}
    u>r\})-\P\l(\l\{\max_{1\le i\le N}\iss {A_iR_i}
    u>r\r\}\r)\r)r^{\c+\be-1}dr\le N\int_0^1r^{\c+\be-1}dr<\8.
    \end{align*}
Let us define $\overline{F}(y)=\P(\{\iss {AR}
    u> y\})$, and $\ga=\c+\be-\be_1$, and notice $N\P(\{\iss {AR}
    u>r\})-\P\l(\l\{\max_{1\le i\le N}\iss {A_iR_i}
    u>r\r\}\r)=(1-\overline{F}(r))^N-1+N\overline{F}(r)\le e^{-N\overline{F}(r)}-1+N\overline{F}(r)$,
    and for some $c>0$
    $$\overline{F}(r)=\P(\{\iss {AR}
    u>r\})\le r^{-\ga}\E\l(\iss {AR}
    u^{\ga}\r)\le cr^{-\ga}.$$
Clearly, $1<\frac{\chi +\beta }{\gamma }$, and $\be_1<\c/2$ implies  $\ga=\c+\be-\be_1\ge\c/2+\be/2$, hence $\frac{\chi +\beta }{\gamma }<2$. Then
\begin{align*}
    \int_1^{\8}&\l(N\P(\{\iss {AR}
    u>r\})-\P\l(\l\{\max_{1\le i\le N}\iss {A_iR_i}
    u>r\r\}\r)\r)r^{\c+\be-1}dr\\
&\le\int_1^{\8}\l(e^{-N\overline{F}(r)}-1+N\overline{F}(r)\r)r^{\c+\be-1}dr\le
\int_1^{\8}\l(e^{-cNr^{-\ga}}-1+cNr^{-\ga}\r)r^{\c+\be-1}dr\\
&=\int_1^{\8}\l(e^{-cNr^{-\ga}}-1+cNr^{-\ga}\r)\l(\l(cNr^{-\ga}\r)\frac{1}{cN}\r)^{-\frac{\c+\be}{\ga}}\frac{dr}{r}\\
&=\frac{(cN)^{\frac{\c+\be}{\ga}}}{\ga}\int_0^{cN}(e^{-r}-1+r)r^{-\frac{\c+\be}{\ga}-1}dr\le\frac{(cN)^{\frac{\c+\be}{\ga}}}{\ga}\int_0^{\8}(e^{-r}-1+r)r^{-\frac{\c+\be}{\ga}-1}dr\\
&\le\frac{(cN)^{\frac{\c+\be}{\ga}}}{\ga}\l(\frac 12\int_0^{1}r^{1-\frac{\c+\be}{\ga}}dr+
\int_{1}^{\8}r^{-\frac{\c+\be}{\ga}}dr\r)=\frac{(cN)^{\frac{\c+\be}{\ga}}}{\ga}\l(\frac{1}{2\l(2-\frac{\c+\be}{\ga}\r)}+\frac{1}{\frac{\c+\be}{\ga}-1}\r)<\8.
\end{align*}
We have shown that $I_1\le C_{\be}e^{-\be t}$, for every $\be\in[0, \be_1)$ with the constant
$C_{\be}>0$ which does not depend on $u\in\Sp$. A straightforward applications of Fubini theorem
yields

\begin{align*}
    \int_{0}^{\8}&\l(N\P(\{\iss {AR}
    u>r\})-\P\l(\l\{\max_{1\le i\le N}\iss {A_iR_i}
    u>r\r\}\r)\r)r^{\c+\be-1}dr\\
&=\int_{0}^{\8}\l(\E\l(\sum_{i=1}^N\I{\{\iss{A_iR_i}{u}>r\}}\r)-\E\l(\I{\{\max_{1\le i\le
N}\iss{A_iR_i}{u}>r\}}\r)\r)r^{\c+\be-1}dr\\
&=\E\l(\int_{0}^{\8}\l(\sum_{i=1}^N\I{\{\iss{A_iR_i}{u}>r\}}-\I{\{\max_{1\le i\le
N}\iss{A_iR_i}{u}>r\}}\r)r^{\c+\be-1}dr\r)\\
&=\E\l(\sum_{i=1}^N\int_{0}^{\iss{A_iR_i}{u}}r^{\c+\be-1}dr-\int_{0}^{\max_{1\le i\le
N}\iss{A_iR_i}{u}}r^{\c+\be-1}dr\r)\\
&=\frac{1}{\c+\be}\E\l(\sum_{i=1}^N\iss{A_iR_i}{u}^{\c+\be}-\l(\max_{1\le i\le
N}\iss{A_iR_i}{u}\r)^{\c+\be}\r),
\end{align*}
In order to show the continuity of $\Sp\times\R\ni(u, t)\mapsto g_1(u, t)$ it is enough to prove the continuity of
\begin{align}\label{contg1}
u\mapsto \frac{1}{e^t}\int_{0}^{e^t}r^{\c}\l(N\P(\{\iss {AR}
    u>r\})-\P\l(\l\{\max_{1\le i\le
N}\iss {A_iR_i}
    u>r\r\}\r)\r)dr.
\end{align}
In this purpose observe that $\P\l(\l\{\max_{1\le i\le
N}\iss {A_iR_i}
    u>r\r\}\r)=1-\l(1-\P(\{\iss {AR}
    u>r\})\r)^N$, hence Lemma \ref{contP} guarantees that
    $$u\mapsto N\P(\{\iss {AR}
    u>r\})-\P\l(\l\{\max_{1\le i\le
N}\iss {A_iR_i}
    u>r\r\}\r),$$ is continuous. Observe that
\begin{multline*}
        N\P(\{\iss {AR}
    u>r\})-\P\l(\l\{\max_{1\le i\le
N}\iss {A_iR_i}
    u>r\r\}\r)\le\begin{cases}
N, & \mbox{if $r\le1$},\\
e^{-N\overline{F}(r)}-1+N\overline{F}(r), & \mbox{if $r>1$},
    \end{cases}
\end{multline*}
then arguing in a similar way as above with $\be=0$, and using Lebesgue dominated convergence
theorem we obtain the continuity of \eqref{contg1} and the lemma follows.
\end{proof}

Now we are going to prove inequality \eqref{nier11} and
\eqref{nier22}, that will provide necessary estimates for Lemma
\ref{fung2lem}. The first one was proved in \cite{gaJel1} and was
sufficient in the one dimensional case discussed there. The second
one is more subtle and allows us to deal with our situation.

\begin{lem}\label{nier1}
Let $\al>1$ and $p=\lceil\al\rceil\ge2$. For any sequence of nonnegative i.i.d. random variables
$Y, Y_1, Y_2,\ldots$ such that $\E(Y^{p-1})<\8$, and any $k\in\N$ we have
\begin{align}\label{nier11}
    \E\l(\l(\sumi 1 k Y_i\r)^{\al}-\sumi i k Y_i^{\al}\r)\le
    k^{\al}\E\l(Y^{p-1}\r)^{\frac{\al}{p-1}}.
\end{align}
\end{lem}
\begin{proof}
As mentioned before the proof is contained in \cite{gaJel1}.
\end{proof}
\begin{lem}\label{nier2}
Let $p\in\N$ and $\be\in(0, 1)$. Then for any $\de\in\l(0, \frac{p(1-\be)}{p+1}\r)$, for any
sequence of nonnegative i.i.d. random variables $Y, Y_1, Y_2,\ldots$ such that $\E(Y^{p-\de})<\8$,
and any $k\in\N$ we have
\begin{align}\label{nier22}
    \E\l(\l(\sumi 1 k Y_i\r)^{p+\be}-\sumi 1 k Y_i^{p+\be}\r)&\le k^{p+1}
    \E\l(Y^{p-\de}\r)^{\frac{p+\be}{p-\de}}.
\end{align}
\end{lem}
\begin{proof}  Define $A_p(k)=\{(j_1,\ldots,j_k)\in\Z^k:
j_1+\ldots+j_k=p,\mbox{ and } 0\le j_i<p\}$ and observe that
\begin{align*}
    \l(\sumi 1 k Y_i\r)^{p-\de}&=\l(\l(\sumi 1 k Y_i\r)^p\r)^\frac{p-\de}{p}\\
&=\l(\sumi 1 k Y_i^p+\sum_{(j_1,\ldots,j_k)\in A_p(k)}\binom{p}{j_1,\ldots,j_k}Y_1^{j_1}\raa
    Y_k^{j_k}\r)^\frac{p-\de}{p}\\
    &\le\sumi 1 k Y_i^{p-\de}+\sum_{(j_1,\ldots,j_k)\in A_p(k)}\binom{p}{j_1,\ldots,j_k}\l(Y_1^{j_1}\raa
    Y_k^{j_k}\r)^\frac{p-\de}{p}.
\end{align*}
Now observe that $\be+\de<\be+\frac{p(1-\be)}{p+1}<1$. By the
above inequality
\begin{align*}
    \l(\sumi 1 k Y_i\r)^{p+\be}&=\l(\sumi 1 k Y_i\r)^{p-\de}\l(\sumi 1 k Y_i\r)^{\be+\de}=\l(\l(\sumi 1 k Y_i\r)^{p-\de}-\sumi 1 k Y_i^{p-\de}\r)\l(\sumi 1 k
    Y_i\r)^{\be+\de}\\
&+\l(\sumi 1 k Y_i^{p-\de}\r)\l(\sumi 1 k Y_i\r)^{\be+\de}\\
&\le\l(\sum_{(j_1,\ldots,j_k)\in A_p(k)}\binom{p}{j_1,\ldots,j_k}\l(Y_1^{j_1}\raa
    Y_k^{j_k}\r)^\frac{p-\de}{p}\r)\l(\sumi 1 k Y_i^{\be+\de}\r)\\
    &+\l(\sumi 1 k Y_i^{p-\de}\r)\l(\sumi 1 k Y_i^{\be+\de}\r).
\end{align*}
It follows that
\begin{align*}
    \l(\sumi 1 k Y_i\r)^{p+\be}-\sumi 1 k Y_i^{p+\be}&\le\sumi 1 k\sum_{(j_1,\ldots,j_k)\in A_p(k)}\binom{p}{j_1,\ldots,j_k}\l(Y_1^{j_1}\raa
    Y_k^{j_k}\r)^\frac{p-\de}{p}Y_i^{\be+\de}\\
    &+\sum_{i\not=j}Y_i^{p-\de}Y_j^{\be+\de}.
\end{align*}
But  $j_i\leq p-1$. Hence
 $\frac{j_i(p-\de)}{p}+\be+\de\le\frac
1p(p-1)(p-\de)+\be+\de
\leq p+\be-1+\frac\de p<p-\de$, since
$$0<\de<\frac{p(1-\be)}{p+1}\ \Longrightarrow\ \de\l(1+\frac 1p\r)<1-\be\ \Longrightarrow\ \be-1+\frac{\de}{p}<-\de \Longrightarrow\ p+\be-1+\frac{\de}{p}<p-\de.$$
Now we have
\begin{align*}
    \E\l(Y_1^{\frac{j_1(p-\de)}{p}}\raa Y_i^{\frac{j_i(p-\de)}{p}+\be+\de}\raa
    Y_k^{\frac{j_k(p-\de)}{p}}\r)&\le\|Y\|^{\frac{j_1(p-\de)}{p}}_{p-\de}\raa\|Y\|_{p-\de}^{\frac{j_i(p-\de)}{p}+\be+\de}\raa\|Y\|_{p-\de}^{\frac{j_k(p-\de)}{p}}\\
    &=\|Y\|_{p-\de}^{p+\be},
\end{align*}
because $j_1+...+j_k=p$. Observe that
$$\de<\frac{p(1-\be)}{p+1}\ \Longrightarrow\ \de<\frac{p-\be}{2}\Longrightarrow\ \be+\de<p-\de,$$
hence
$\E\l(Y_i^{p-\de}Y_j^{\be+\de}\r)=\|Y\|_{p-\de}^{p-\de}\|Y\|^{\be+\de}_{\be+\de}\le\|Y\|_{p-\de}^{p+\be},$
and so
\begin{align*}
    \E\l(\l(\sumi 1 k Y_i\r)^{p+\be}-\sumi 1 k Y_i^{p+\be}\r)&\le
    k(k^p-k)\E\l(Y^{p-\de}\r)^{\frac{p+\be}{p-\de}}+k^2\E\l(Y^{p-\de}\r)^{\frac{p+\be}{p-\de}}=k^{p+1}\E\l(Y^{p-\de}\r)^{\frac{p+\be}{p-\de}}.
\end{align*}
\end{proof}

\begin{lem}\label{fung2lem}
Under the assumptions of Theorem \ref{thm2}, there exists $0<\be_2<1$ such that for every
$\be\in[0, \be_2)$, there is a finite constant $C_{\be}>0$, such that for every $(u,
t)\in\Sp\times\R$ we have
\begin{multline}\label{fung2}
    g_2(u, t)=\frac{1}{e^te^{\c}(u)}\int_{0}^{e^t}r^{\c}\l|\P(\{\iss R u>r\})-\P\l(\l\{\max_{1\le i\le N}\iss {A_iR_i}
    u>r\r\}\r)\r|dr\\
    \le C_{\be}e^{-\be|t|},
\end{multline}
and
\begin{multline}\label{fung2a}
    \int_{0}^{\8}r^{\c+\be-1}\l(\P(\{\iss R u>r\})-\P\l(\l\{\max_{1\le i\le N}\iss {A_iR_i}
    u>r\r\}\r)\r)dr\\
    =\frac{1}{\c+\be}\E\l(\iss Ru^{\c+\be}-\l(\max_{1\le i\le N}\iss {A_iR_i}
    u\r)^{\c+\be}\r).
\end{multline}
Moreover,
$\Sp\times\R\ni(u, t)\mapsto g_2(u, t)$ is continuous.
\end{lem}
\begin{proof}
Let $0<\be_2<\min\{\eps, \be_1\}$ ($\eps>0$ as in Theorem \ref{thm2} and $\be_1>0$ as in Lemma \ref{fung1lem}) and take $\be\in[0, \be_2)$. Then
for every $t>0$
\begin{align*}
    I_2&=e^{-\be t}e^{-(1-\be)t}\int_{0}^{e^t}r^{\c}\l|\P(\{\iss R u>r\})-\P\l(\l\{\max_{1\le i\le N}\iss {A_iR_i}
    u>r\r\}\r)\r|dr\\
    &\le e^{-\be t}\int_{0}^{\8}r^{\c+\be-1}\l|\P(\{\iss R u>r\})-\P\l(\l\{\max_{1\le i\le N}\iss {A_iR_i}
    u>r\r\}\r)\r|dr.
\end{align*}
Observe that $\iss Ru\ge\max_{1\le i\le N}\iss{A_iR_i}{u}$. Then applying Fubini theorem as in
Lemma \ref{fung1lem} we obtain
\begin{align*}
    \int_{0}^{\8}&r^{\c+\be-1}\l(\P(\{\iss R u>r\})-\P\l(\l\{\max_{1\le i\le N}\iss {A_iR_i}
    u>r\r\}\r)\r)dr\\
  &=\frac{1}{\c+\be}\E\l(\iss Ru^{\c+\be}-\l(\max_{1\le i\le N}\iss {A_iR_i}
    u\r)^{\c+\be}\r).
\end{align*}
If $0<\c<1$, take any $\be\in[0, \be_2)$ such that $0<\c+\be\le1$ and notice
\begin{align*}
    \E&\l(\iss Ru^{\c+\be}-\l(\max_{1\le i\le N}\iss {A_iR_i}
    u\r)^{\c+\be}\r)\\
    &\le\E\l(\iss Bu^{\c+\be}\r)+\E\l(\sum_{i=1}^N\iss {A_iR_i}
    u^{\c+\be}-\l(\max_{1\le i\le N}\iss {A_iR_i}
    u\r)^{\c+\be}\r)<\8,
\end{align*}
since $\E(|B|^{\c+\varepsilon})<\8$ for some $\eps>0$, and second term is finite by Lemma
\ref{fung1lem}. \\
\noindent If $\c\ge1$ we write
\begin{align*}
    \E\l(\iss Ru^{\c+\be}-\l(\max_{1\le i\le N}\iss {A_iR_i}
    u\r)^{\c+\be}\r)
    &=\E\l(\iss Ru^{\c+\be}-\sum_{i=1}^N\iss {A_iR_i}
    u^{\c+\be}\r)\\
    &+\E\l(\sum_{i=1}^N\iss {A_iR_i}
    u^{\c+\be}-\l(\max_{1\le i\le N}\iss {A_iR_i}
    u\r)^{\c+\be}\r).
\end{align*}
We have to estimate only the first term, since the second one is
finite by Lemma \ref{fung1lem}. In this purpose we use Lemma
\ref{nier1} and \ref{nier2}. Notice that
\begin{align*}
    \E\l(\iss Ru^{\c+\be}-\sum_{i=1}^N\iss {A_iR_i}
    u^{\c+\be}\r)&=\E\l(\iss {\sum_{i=1}^NA_iR_i+B}u^{\c+\be}-\iss {\sum_{i=1}^NA_iR_i}u^{\c+\be}\r)\\
    &+\E\l(\iss {\sum_{i=1}^NA_iR_i}
    u^{\c+\be}-\sum_{i=1}^N\iss {A_iR_i}
    u^{\c+\be}\r)\\
    &\le(\c+\be)\E\l(|B|\l(\sum_{i=1}^N|A_iR_i|+|B|\r)^{\c+\be-1}\r)\\
    &+\E\l(\iss {\sum_{i=1}^NA_iR_i}
    u^{\c+\be}-\sum_{i=1}^N\iss {A_iR_i}
    u^{\c+\be}\r).
\end{align*}
$\E\l(|B|\l(\sum_{i=1}^N|A_iR_i|+|B|\r)^{\c+\be-1}\r)$ is finite,
since $\E(\|A\|^{\c+\be-1})<\8$, $\E(|B|^{\c+\eps})<\8$ and Theorem \ref{thm1} yields
$\E(|R|^{\c+\be-1})<\8$.

If $\c\not\in\N$ we assume additionally that $\lceil\c+\be_2\rceil=\lceil\c\rceil$, (which holds
for sufficiently small $\be_2>0$). Applying inequality \eqref{nier11} with
$p=\lceil\c\rceil=\lceil\c+\be\rceil$ and $\be\in[0, \be_2)$ we obtain
\begin{align*}
    \E\l(\iss {\sum_{i=1}^NA_iR_i}
    u^{\c+\be}-\sum_{i=1}^N\iss {A_iR_i}
    u^{\c+\be}\r)\le N^{\c+\be}\l(\E\l(\iss {AR}
    u^{p-1}\r)\r)^{\frac{\c+\be}{p-1}}<\8,
\end{align*}
since $p-1<\c$.

If $\c\in\N$ and $\be\in[0, \be_2)$ take any $\de\in\l(0, \frac{p(1-\be)}{p+1}\r)$ as in Lemma
\ref{nier2} with $p=\c$, then by inequality \eqref{nier22} we get
\begin{align*}
    \E\l(\iss {\sum_{i=1}^NA_iR_i}
    u^{\c+\be}-\sum_{i=1}^N\iss {A_iR_i}
    u^{\c+\be}\r)\le N^{\c+1}\l(\E\l(\iss {AR}
    u^{\c-\de}\r)\r)^{\frac{\c+\be}{\c-\de}}<\8.
\end{align*}
Finally, we have proved $I_2\le C_{\be}e^{-\be|t|}$, for every $\be\in[0, \be_2)$ with $C_{\be}<\8$
independent of $u\in\Sp$.

It remains to prove that $\Sp\times\R\ni(u, t)\mapsto g_2(u, t)$ is continuous. In this purpose it suffices to show continuity of
\begin{align}\label{contg2}
u\mapsto \frac{1}{e^t}\int_{0}^{e^t}r^{\c}\l(\P(\{\iss {R}
    u>r\})-\P\l(\l\{\max_{1\le i\le
N}\iss {A_iR_i}
    u>r\r\}\r)\r)dr.
\end{align}
Observe that
    \begin{multline*}
        \frac{1}{e^t}\int_{0}^{e^t}r^{\c}\Bigg|\P(\{\iss {R}
    {u_n}>r\})-\P\l(\l\{\max_{1\le i\le
N}\iss {A_iR_i}
    {u_n}>r\r\}\r)\\
    -\l(\P(\{\iss {R}
    {u_0}>r\})-\P\l(\l\{\max_{1\le i\le
N}\iss {A_iR_i}
    {u_0}>r\r\}\r)\r)\Bigg|dr\\
    \le\int_{0}^{\8}r^{\c-1}\Bigg|\P(\{\iss {R}
    {u_n}>r\})-\P\l(\l\{\max_{1\le i\le
N}\iss {A_iR_i}
    {u_n}>r\r\}\r)\\
    -\l(\P(\{\iss {R}
    {u_0}>r\})-\P\l(\l\{\max_{1\le i\le
N}\iss {A_iR_i}
    {u_0}>r\r\}\r)\r)\Bigg|dr.
    \end{multline*}
It is enough to show that the last integral converges to $0$ as $\limn u_n=u_0$.
In this purpose we will use an extended version of Lebesgue dominated convergence theorem (see for instance in \cite{gaBo}). Namely,
\begin{thm}\label{Leb}
    Given a measure space $(X, \mathcal{M}, \mu)$ (where $\mu$ may takes values in $[0, \8]$). Let $(f_n)_{n\in\N}$ and $(h_n)_{n\in\N}$, $f$ and $h$ be $\mathcal{M}$ measurable, real valued functions on $X$. Suppose
    \begin{itemize}
      \item $\limn f_n=f$ and $\limn h_n=h$ a.e. on $X$,
      \item $(h_n)_{n\in\N}$ and $h$ are all $\mu$ integrable on $X$ and $\limn\int_X h_nd\mu=\int_Xhd\mu$,
      \item $|f_n|\le h_n$ a.e. on $X$ for every $n\in\N$.
    \end{itemize}
    Then $f$ is $\mu$ integrable on $X$ and $\limn\int_X f_nd\mu=\int_Xfd\mu$.
\end{thm}
We will apply Theorem \ref{Leb} with
\begin{multline*}
    f_n(r)=r^{\c-1}\Bigg|\P(\{\iss {R}
    {u_n}>r\})-\P\l(\l\{\max_{1\le i\le
N}\iss {A_iR_i}
    {u_n}>r\r\}\r)\\
    -\l(\P(\{\iss {R}
    {u_0}>r\})-\P\l(\l\{\max_{1\le i\le
N}\iss {A_iR_i}
    {u_0}>r\r\}\r)\r)\Bigg|,
\end{multline*}
\begin{multline*}
    h_n(r)=r^{\c-1}\Bigg(\P(\{\iss {R}
    {u_n}>r\})-\P\l(\l\{\max_{1\le i\le
N}\iss {A_iR_i}
    {u_n}>r\r\}\r)\\
    +\l(\P(\{\iss {R}
    {u_0}>r\})-\P\l(\l\{\max_{1\le i\le
N}\iss {A_iR_i}
    {u_0}>r\r\}\r)\r)\Bigg),
\end{multline*}
and
\begin{align*}
    h(r)=2r^{\c-1}\l(\P(\{\iss {R}
    {u_0}>r\})-\P\l(\l\{\max_{1\le i\le
N}\iss {A_iR_i}
    {u_0}>r\r\}\r)\r).
\end{align*}
Clearly, $|f_n|\le h_n$ for every $n\in\N$, and by the previous part of the lemma $(h_n)_{n\in\N}$ and $h$ are all integrable. Lemma \ref{contP} guarantees that $\limn f_n(r)=0$ and $\limn h_n(r)=h(r)$. In order to show that $\limn\int_0^{\8} h_n(r)dr=\int_0^{\8}h(r)dr$, notice that by \eqref{fung2a} with $\be=0$ we have to show that \begin{align}\label{ident}
    \limn\E\l(\iss R{u_n}^{\c}-\l(\max_{1\le i\le N}\iss {A_iR_i}
    {u_n}\r)^{\c}\r)=\E\l(\iss R{u_0}^{\c}-\l(\max_{1\le i\le N}\iss {A_iR_i}
    {u_0}\r)^{\c}\r).
\end{align}
But in view of the first part of the lemma and the estimates given there \eqref{ident} is a simple consequence of a classical Lebesgue dominated convergence theorem.
This finishes the proof of Lemma \ref{fung2lem}.
\end{proof}

\begin{proof}[Proof of Theorem \ref{thm2}] From Lemma \ref{funG0lem} we know that
\begin{align*}
    G(u, t)=\sum_{n=0}^{\8}\Te^{n}g(u, t),
\end{align*}
where
\begin{align*}
    g(u, t)=\frac{1}{e^te^{\c}(u)}\int_{0}^{e^t}r^{\c}\l(\P(\{\iss R u>r\})-N\P(\{\iss {AR}
    u>r\})\r)dr.
\end{align*}
As a consequence of Lemma \ref{fung1lem} and Lemma \ref{fung2lem} the function $\Sp\times\R\ni(u,
t)\mapsto g(u, t)$  is jointly continuous. Moreover, it is possible to find $\be>0$ and a positive
constant $C_{\be}<\8$ such that
\begin{align*}
|g(u, t)|\le C_{\be}e^{-\be|t|},\ \ \ \mbox{for every $(u, t)\in\Sp\times\R$},
\end{align*}
since $|g(u, t)|\le g_1(u, t)+g_2(u, t)$, for $g_1(u, t)$ and $g_2(u, t)$ defined in Lemma
\ref{fung1lem} and Lemma \ref{fung2lem} respectively.
This shows that $g(u, t)$ satisfies condition \eqref{dri}. By the Kesten's renewal theorem
\ref{kesten} we obtain
\begin{align*}
    \lim_{t\to\8} G(u, t)=\lim_{t\to\8} \E_x^{\c,*} \l(\sum_{n=0}^\8 g(X_n,t-V_n)\r) =\frac{1}{\al(\c)}\int_{\Sp}\l(\int_{\R}g(y, x)dx\r)\pi^{\c}_*(dy)=C_{\c}.
\end{align*}
In other words we have proved that for every $u\in\Sp$
\begin{align*}
    \lim_{t\to\8}G(u, t)=\lim_{t\to\8}\frac{1}{e^te^{\c}(u)}\int_{0}^{e^t}r^{\c}\P(\{\iss R u>r\})dr=
    C_{\c}\ge 0.
\end{align*}
Hence in view of Lemma 9.3 of \cite{G}, for every $u\in\Sp$
\begin{align*}
    \lim_{t\to\8}t^{\c}\P(\{\iss R u>t\})= C_{\c}e^{\c}(u).
\end{align*}
It remains to prove that $C_{\c}>0$ for every $\c\ge1$. In this purpose notice that
\begin{multline*}
    C_{\c}=\frac{1}{\al(\c)}\int_{\Sp}\l(\int_{\R}g(u, t)dt\r)\pi^{\c}_*(du)\\
    =\frac{1}{\al(\c)}\int_{\Sp}\int_{\R}\l(\frac{1}{e^te^{\c}(u)}\int_{0}^{e^t}r^{\c}(\P(\{\iss R u>r\})-N\P(\{\iss {AR}u>r\}))dr\r)dt\pi^{\c}_*(du)\\
    =\frac{1}{\al(\c)}\int_{\Sp}\int_{\R}\l(\frac{1}{e^te^{\c}(u)}\int_{-\8}^{t}e^{s(\c+1)}(\P(\{\iss R u>e^s\})-N\P(\{\iss {AR}u>e^s\}))ds\r)dt\pi^{\c}_*(du)\\
    =\frac{1}{\al(\c)}\int_{\Sp}\int_{\R}\int_{s}^{\8}\l(\frac{e^{s(\c+1)}}{e^te^{\c}(u)}(\P(\{\iss R u>e^s\})-N\P(\{\iss {AR}u>e^s\}))dt\r)ds\pi^{\c}_*(du)=\\
    \end{multline*}
    \begin{multline*}
    =\frac{1}{\al(\c)}\int_{\Sp}\int_{\R}\frac{e^{s\c}}{e^{\c}(u)}(\P(\{\iss R u>e^s\})-N\P(\{\iss {AR}u>e^s\}))ds\pi^{\c}_*(du)\\
    =\frac{1}{\al(\c)}\int_{\Sp}\frac{1}{e^{\c}(u)}\int_{0}^{\8}r^{\c-1}(\P(\{\iss R u>r\})-N\P(\{\iss {AR}u>r\}))dr\pi^{\c}_*(du)\\
    =\frac{1}{\al(\c)}\int_{\Sp}\frac{1}{e^{\c}(u)}\int_{0}^{\8}
    \E\l(\mathbf{1}_{\l\{\iss{\sum_{i=1}^NA_iR_i+B}{u}>r\r\}}
    -\sum_{i=1}^N\mathbf{1}_{\l\{\iss{A_iR_i}{u}>r\r\}}\r)r^{\c-1}dr\pi^{\c}_*(du)\\
    =\frac{1}{\al(\c)\c}\int_{\Sp}\frac{1}{e^{\c}(u)}\E\l(\iss{\sum_{i=1}^NA_iR_i+B}{u}^{\c}
    -\sum_{i=1}^N\iss{A_iR_i}{u}^{\c}\r)\pi^{\c}_*(du)\\
    \ge\frac{1}{\al(\c)\c}\int_{\Sp}\frac{1}{e^{\c}(u)}\E\l(\iss{B}{u}^{\c}\r)\pi^{\c}_*(du),
\end{multline*}
since we have used
\begin{align*}
    \l(\sum_{i=1}^N\iss{A_iR_i}{u}^{\c}+\iss{B}{u}^{\c}\r)^{1/\c}\le
    \sum_{i=1}^N\iss{A_iR_i}{u}+\iss{B}{u}.
\end{align*}
We need only to show that
\begin{align}\label{ostdow1}
   \int_{\Sp}\frac{1}{e^{\c}(u)}\E\l(\iss{B}{u}^{\c}\r)\pi^{\c}_*(du)>0.
\end{align}
We will show that there exists $c_{\c}>0$ such that
\begin{align}\label{ostdow2}
   \int_{\Sp}\iss{x}{u}^{\c}\pi^{\c}_*(du)\ge c_{\c}\|x\|^{\c},
\end{align}
for every $x\in\R^d_+$. Observe that $\Sp\ni x\mapsto\int_{\Sp}\iss{x}{u}^{\c}\pi^{\c}_*(du)$ is continuous and nonzero for every $x\in\Sp$, since $\supp\pi^{\c}_*$ is not contained in any proper subspace of $\Sp$ (see Section\eqref{TO}). This allows us to conclude that $ x\mapsto\int_{\Sp}\iss{x}{u}^{\c}\pi^{\c}_*(du)$ attains its minimum $c_{\c}>0$ on $\Sp$, and in fact this proves \eqref{ostdow2}.

In order to prove \eqref{ostdow1} notice that by \eqref{ostdow2} we obtain
\begin{multline*}
\int_{\Sp}\frac{1}{e^{\c}(u)}\E\l(\iss{B}{u}^{\c}\r)\pi^{\c}_*(du)\\
\ge\frac{1}{\sup_{u\in\Sp}e^{\c}(u)}\int_{\Sp}\E\l(\iss{B}{u}^{\c}\r)\pi^{\c}_*(du)\\
\ge\frac{1}{\sup_{u\in\Sp}e^{\c}(u)}\E\l(\int_{\Sp}\iss{B}{u}^{\c}\pi^{\c}_*(du)\r)\\
\ge\frac{c_{\c}}{\sup_{u\in\Sp}e^{\c}(u)}\E\l(\|B\|^{\c}\r)>0,
\end{multline*}
since $\P(\{B>0\})>0$. This completes the proof of Theorem \ref{thm2}.
\end{proof}

\end{document}